\documentclass[11pt]{article}
\usepackage{amsmath}
\usepackage{amsthm}
\usepackage{amsfonts}
\usepackage{amssymb}
\usepackage{color}
\usepackage{graphicx}
\usepackage{url}

\newcommand{\mbb}[1]{\mathbb{#1}}
\newcommand{\R}{\mbb{R}}

\usepackage[top=2.5cm, bottom=3.5cm, left=2.5cm, right=2.5cm]{geometry} 
\renewcommand{\epsilon}{\varepsilon}
\DeclareMathOperator{\Ima}{Im}

\newcommand{\norm}[1]{\ensuremath{\left\|#1\right\|}}
\newcommand{\bdr}[1]{bdr\ #1}
\newtheorem{theorem}{Theorem}[section]

\newtheorem{remark}[theorem]{Remark}

\usepackage{mathtools}
\DeclarePairedDelimiter\floor{\lfloor}{\rfloor}

\usepackage[noend]{algorithmic}
\usepackage{algorithm}
\floatname{algorithm}{Algorithm}

\title{Rigorous cubical approximation and persistent homology \\
       of continuous functions}

\author{Pawe{\l} D{\l}otko \\
        Inria Saclay -- Ile-de-France \\
        1 rue Honor\'e d'Estienne d'Orves \\
        91120 Palaiseau, France
  \and
        Thomas Wanner \\
        Department of Mathematical Sciences \\
        George Mason University \\
        Fairfax, VA 22030, USA }

\date{ \today}

\begin{document}         
\maketitle
\begin{abstract}
The interaction between discrete and continuous mathematics lies at the heart
of many fundamental problems in applied mathematics and computational sciences.
In this paper we discuss the problem of discretizing vector-valued functions 
defined on finite-dimensional Euclidean spaces in such a way that the discretization
error is bounded by a pre-specified small constant. While the approximation
scheme has a number of potential applications, we consider its usefulness in
the context of computational homology. More precisely, we demonstrate that our
approximation procedure can be used to rigorously compute the persistent
homology of the original continuous function on a compact domain, up to small
explicitly known and verified errors. In contrast to other work in this area,
our approach requires minimal smoothness assumptions on the underlying function.
\end{abstract}
\tableofcontents
\newpage
\section{Introduction}
\label{sec:introduction}
In many applied situations it is necessary to replace a given smooth function
by a discretized version which in some sense is close to the given original
mapping. In this paper we discuss a number of algorithms for approximating a
continuous function $f : D \rightarrow \mathbb{R}^m$, defined on a compact
rectangular domain $D \subset \mathbb{R}^n$, using piecewise constant
functions. In order to achieve approximations with mathematically verified
approximation bounds, our approach relies on the use of rigorous computer
arithmetic. One such example is interval arithmetic~\cite{moore}, and it will
form the foundation of the algorithms described in the present paper. However,
our approach can also be used with any other rigorous computer arithmetic which
provides rigorous enclosures of function values. For example, our algorithms
can readily be adapted to employ alternative tools such as affine
arithmetic~\cite{affineArithmetic}, generalized interval
arithmetic~\cite{generalizedIntervalArithmetic}, or in fact,
any other rigorous arithmetic capable of providing range enclosures.

Upon successful completion, the algorithms presented in this paper will
return a decomposition of the rectangular domain~$D$ into a collection of
compact $n$-dimensional rectangles, any two of which only intersect in a
subset of their topological boundaries. Furthermore, each rectangle~$R$
in the collection is assigned a value~$val(R)$ in~$\mathbb{R}^m$. The
approximating function~$\Box f$ is then defined in the following way.
For every rectangle~$R$ in the constructed decomposition, the function
value of~$\Box f$ in the interior of~$R$ is given by~$val(R)$. On the other
hand, if~$x \in D$ lies in the intersection of~$k$ rectangles~$R_1, \ldots,
R_k$, then we define~$\Box f(x) = \min_{i = 1,\ldots,k} val(R_i)$. In this
way, the decomposition of~$D$ produced upon successful completion
of the algorithm gives rise to a lower semi-continuous piecewise constant
approximation~$\Box f$ of the given continuous function~$f$. We will see
later that the algorithm in fact guarantees the inequality
\begin{displaymath}
  \| f - \Box f \|_{\infty} < \epsilon \; ,
  \quad\mbox{ where $\epsilon > 0$ denotes the input parameter of the
  algorithm,}
\end{displaymath}
and where~$\| \cdot \|_\infty$ denotes the usual maximum norm. This procedure
is illustrated in the two panels of Figure~\ref{fig:introex}. In the left
panel, the approximating function~$\Box f$ is shown for the one-dimensional
function~$f(x) = 2 - 25 x + 108 x^2 - 162 x^3 + 81 x^4$ on the
interval~$D = [0,1]$, and with approximation parameter~$\epsilon = 1/2$.
The right panel contains the approximation for the so-called Ackley
function, which is one of the standard test functions in nonlinear
optimization and which will be discussed in more detail later in
this paper.
\begin{figure}[tb]
  \centering
  \includegraphics[width=7cm]{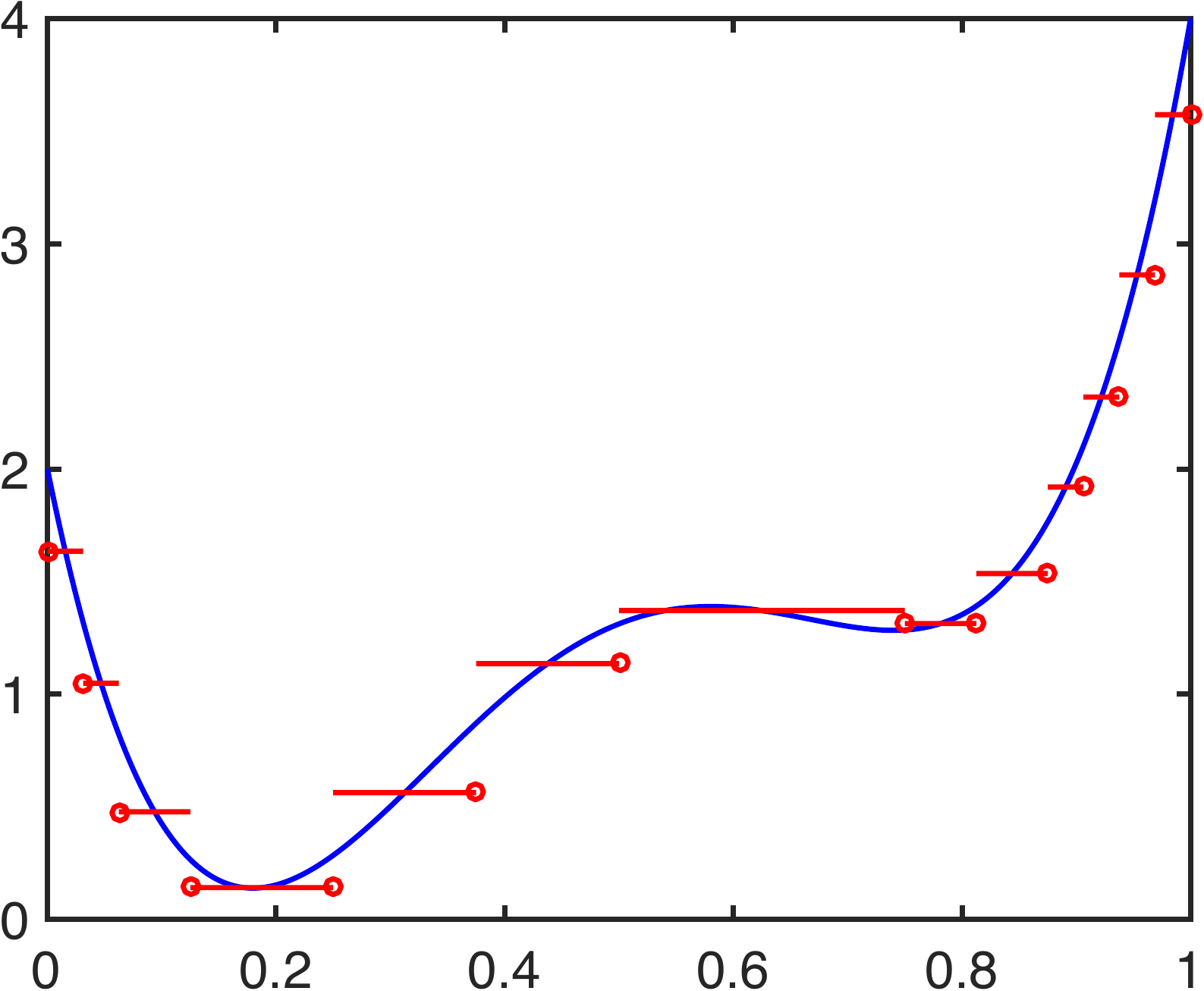}
  \hspace{0.3cm}
  \includegraphics[width=8.2cm]{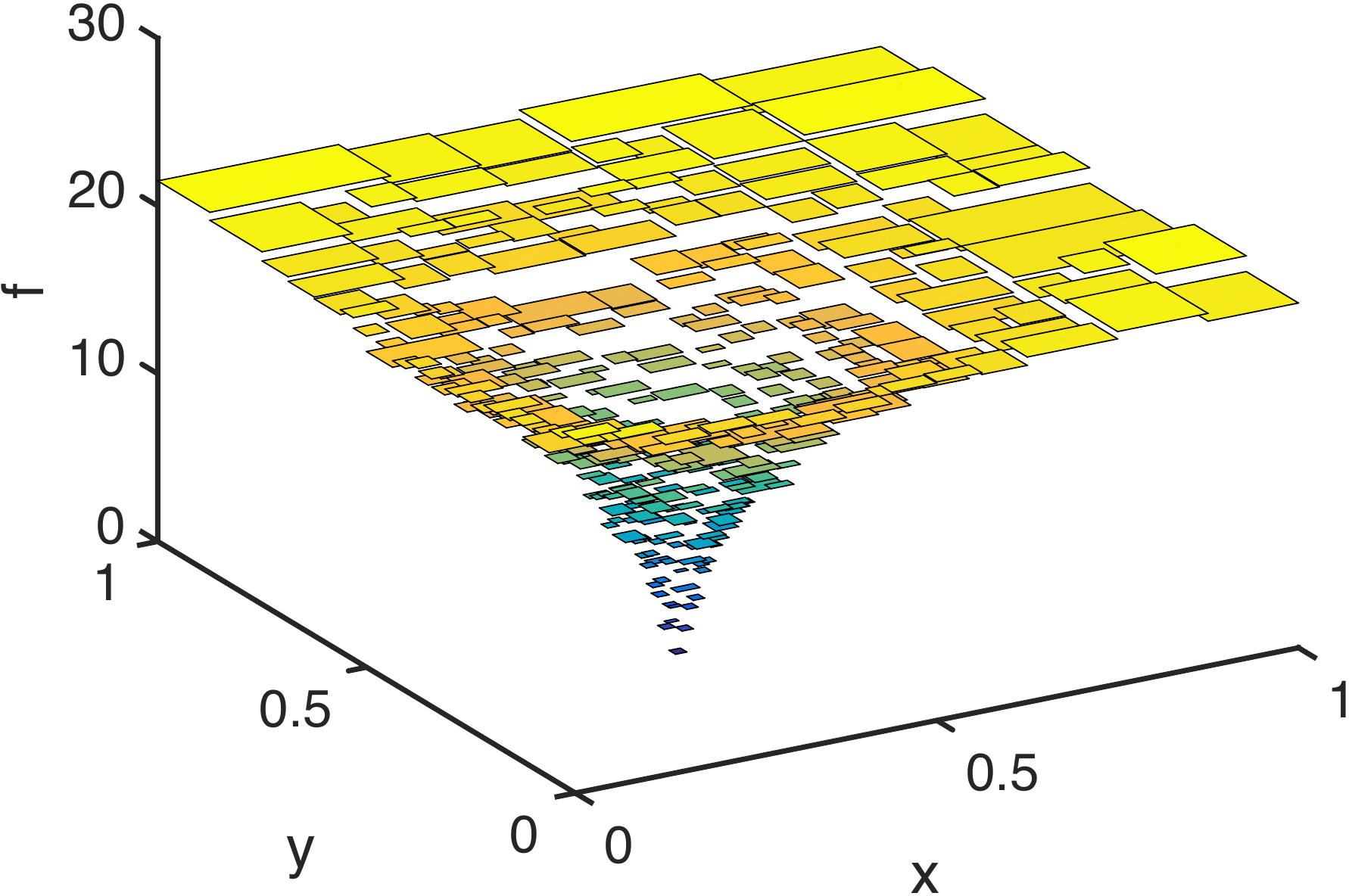}
  \caption{Sample approximating functions~$\Box f$. The left image is
           for the one-dimensional function~$f(x) = 2 - 25 x + 108 x^2
           - 162 x^3 + 81 x^4$ on the interval~$D = [0,1]$, and with
           approximation parameter~$\epsilon = 1/2$. The function~$f$
           is shown in blue, while its approximation~$\Box f$ is depicted
           in red. Red dots indicate the function values at the
           points in~$D$ which are contained in more than one intervals
           of the final decomposition. The right panel shows an approximation
           for the Ackley function with $\epsilon = 2$.}
  \label{fig:introex}
\end{figure}

The approximation algorithms presented in this paper can in principle
be applied in a number of situations. For example, in the context of global 
optimization they allow one to rigorously narrow down the search for global
minima or maxima of a real-valued function, as can easily be seen in the
examples of Figure~\ref{fig:introex}. Rather than pursuing such applications,
we confine ourselves to the computation of the persistent homology of a 
function. One-dimensional persistent homology was introduced
in~\cite{zomorodianPersistence}, and since then has become one of the
main topological tools in the applied sciences for studying the shape
of functions or spaces. Its popularity can be attributed to a variety
of factors:
\begin{itemize}
\item Persistent homology provides quantitative information about the object
      of interest, including information about its shape,
\item the obtained information is both metric and coordinate free, and
\item the information is preserved under continuous deformations, i.e.,
      it is invariant with respect to isometries.
\item Persistence captures topological information at different scales,
\item it is robust with respect to the influences of random noise, and 
\item efficient algorithms are available for its computation.
\end{itemize}
At present, persistent homology can only be determined algorithmically
for the case of scalar-valued functions, i.e., only one-dimensional persistence
is amenable to a computational treatment. For this, one needs to find a suitable
filtration of a finite cell complex, and the results of this paper will enable
its efficient construction. Due to deep results from quiver theory~\cite{gunnar},
there are no finite invariants which completely characterize multi-dimensional
persistent homology. Nevertheless, there are many groups working on providing
partial invariants for multi-dimensional persistent homology that can be
used in practice.

The algorithms presented in this paper for the case of functions $f : D \subset
\mathbb{R}^n \rightarrow \mathbb{R}^m$ provide discretizations of~$f$ with
rigorous error bounds. These discretizations can immediately be used in the
computation of one-dimensional persistent homology, i.e., in the scalar-valued
case $m=1$. We would like to point out, however, that the approximations can
potentially also be used in the context of recently developed algorithms
which compute partial invariants of multi-dimensional persistent homology.
In order to specifically accommodate the computation of persistent homology,
we consider specific approximations in which, in addition to the rectangular
decomposition of the domain~$D$, we also keep track of all the intersections
of the rectangles in the decomposition. This is achieved by constructing a
weighted cell complex which will later be used to compute persistent homology
or some invariants of multi-dimensional persistent homology.

Approximating continuous functions is a part of the classical simplicial
approximation theorem~\cite{introToAT}. However, we are not aware of any
attempts to create lower semi-continuous approximations of continuous
functions using rigorous computer arithmetic. In the context of differentiable
functions, the work~\cite{miro} comes close in spirit to the present paper.
The authors use a generalization of quad-trees and oct-trees to compute
compatible cellular decompositions for level sets of a given differentiable
function. Then, after computing the generators of the associated homology
groups at each level of the filtration, together with their inclusions,
they obtain the persistent homology of a differentiable function. It is
worth mentioning that through the computation of the persistence of
inclusions between two fixed level sets, the authors of~\cite{miro}
obtain an exact persistence result for the chosen levels. Similar idea to compute standard persistent homology was used in~\cite{mw10}. This, however,
requires the successful termination of their algorithm. Unfortunately,
even in the examples presented in their paper, their algorithm cannot
validate a significant number of levels, which is due to grid alignment
issues which were first pointed out in~\cite{cochran:etal:13a}. They arise
naturally in the setting of~\cite{miro}, since the authors have to use binary
subdivisions for the persistence computations, and since varying the threshold
level ensures that the level curves will frequently come close to the binary
subdivision lines. In the context of one fixed level set, this effect has been
one of the shortcomings of the method in~\cite{day:etal:09a}, but it could be
removed through randomized subdivisions in~\cite{cochran:etal:13a}. Unfortunately,
the simple randomization trick cannot be applied in the setting of~\cite{miro}.
Moreover, the total computational cost of the method in~\cite{miro} may in
some cases be higher than the one of our method presented below. This is due
to the fact that in~\cite{miro} the authors need to exactly verify every
intermediate level set, and then perform independent homology computations
at each filtration level. 

Other approaches to approximating the persistence of functions, which also take
errors into account, were presented in~\cite{Edels2} and~\cite{Edels1}. In the
former paper, the authors employ oct-trees to approximate the persistence of a
digital image as accurately as possible for a fixed number of simplification
steps. In fact, one of our algorithms, the one presented in
Section~\ref{sec:greedyApproximation}, is directly inspired by this
approach. Finally, the paper~\cite{Edels1} is concerned with persistent
homology computations in the presence of non-uniform error models. 

The remainder of this paper is organized as follows. In
Section~\ref{sec:PersistentHomology} we give a short introduction to
rectangular CW-complexes, homology and (multi-dimensional) persistent
homology, and we survey necessary results on the stability of persistent
homology. Section~\ref{sec:cubapprox} contains the main results of the
paper. After collecting some basic results from interval arithmetic in
Section~\ref{sec:intervalArithmetic}, we then turn our attention to the
main algorithms. In Section~\ref{sec:approximationContFunctions} we begin
with a method for computing a piecewise constant approximation of a
continuous function~$f : D \subset \mathbb{R}^n \rightarrow \mathbb{R}^m$,
and discuss termination criteria as well as complexity issues for a large
class of functions~$f$. This is followed in Section~\ref{sec:SubdivAlgorithm}
by a rectangle subdivision algorithm which, during the subdivision process, keeps
track of all boundary elements, and which establishes a lower semi-continuous
approximation. In the end, this method allows us to construct
a cell complex which is suitable for the computation of persistence. This will
be outlined in detail in Section~\ref{sec:cubicalApproximationPersistence}, 
where we use the subdivision algorithm to first construct a filtered complex,
and then its (multi-dimensional) persistent homology. It will be shown that the
resulting persistence diagram is close to the persistent homology of the 
underlying function~$f$. A second approximation algorithm, which is motivated
by the approach in~\cite{Edels2}, is the subject of
Section~\ref{sec:greedyApproximation}. This time, we aim to approximate
the persistent homology of~$f$ as effectively as possible using a limited
number of subdivisions. Finally, Section~\ref{sec:experiments} contains a number
of numerical case studies.
\section{Persistent Homology}
\label{sec:PersistentHomology}
In this section we provide a brief introduction to persistent homology,
in order to keep the present paper as self-contained and accessible as
possible. For a more-detailed introduction to standard algebraic topology
we refer the reader to~\cite{massey}, an extensive introduction to applied
and computational topology can be found in~\cite{herbert}.
\subsection{Rectangular CW-Complexes}
\label{sec:PersistentHomology1}
In the following, the term \emph{interval} is always used for a compact
interval $I = [a,b] \subset \mathbb{R}$ with $a \leq b$. We say that the
interval~$I$ is \emph{degenerate} if $a = b$, otherwise it is called
\emph{nondegenerate}. In order to describe the algebraic boundary of
intervals, we say that a degenerate interval has no faces, while the
nondegenerate interval $I = [a,b]$ has the two faces~$[a] := [a,a]$ and
$[b] := [b,b]$.
Intervals as defined above are the building blocks of higher-dimensional
sets in the following sense. A \emph{rectangle} in $\mathbb{R}^d$ is a product
$Q = Q_1 \times Q_2 \times \ldots \times Q_d$ of~$d$ intervals $Q_1, Q_2,
\ldots, Q_d$. The dimension of~$Q$, abbreviated as~$\dim{Q}$, is the number
of nondegenerate intervals among $Q_1, Q_2, \ldots, Q_d$. Consider now
two rectangles $Q = Q_1 \times \ldots \times Q_d$ and $P = P_1 \times
\ldots \times P_d$. Then the rectangle~$P$ is called a \emph{primary
face} of~$Q$ if $P \subset Q$, the dimensions satisfy $\dim{P} =
\dim{Q} - 1$, and if there exists an index~$j$ such that~$P_j$ is a
face of~$Q_j$, in the sense defined above for intervals. Notice that
in this case, the index~$j$ is uniquely determined, and we have
$\dim P_i = \dim Q_i$ for all $i \neq j$. We would like to point out,
however, that it is not required that the identity $P_i = Q_i$ holds
for any of these indices --- only the inclusions $P_i \subset Q_i$
are necessary. Finally, the rectangle~$P$ is called a \emph{face} of~$Q$,
if either there is a descending sequence of primary faces joining~$Q$
to~$P$, or if~$P$ is the empty set.

We now turn to the definition of the boundary of a rectangle. For an
interval $I = [a,b]$, its \emph{boundary} is defined as $\bdr{I} =
\{a,b\}$ if $a \neq b$, and as $\bdr{I} = \emptyset$ if $a > b$. The
boundary of a higher-dimensional rectangle $Q = Q_1 \times \ldots
\times Q_d$ is defined as the union
\begin{displaymath}
  \bdr{Q} = \bigcup_{i=1}^d Q_1 \times \ldots \times \bdr{Q_i}
  \times \ldots \times Q_d \; .
\end{displaymath}
One can easily see that the boundary~$\bdr{Q}$ corresponds to boundary
of the $\dim{Q}$-dimensional manifold~$Q$.

The above-defined rectangles are the building blocks for more general
structured sets. For the purposes of this paper, we consider a
\emph{rectangular CW-complex}, which can be defined as follows.
A {\em rectangular structure\/} is a finite collection~$\mathcal{Q}$
of rectangles in~$\mathbb{R}^d$ such that for any $P, Q \in \mathcal{Q}$
we either have $P \cap Q = \emptyset$, or $P \cap Q$ is a common face of
both~$P$ and~$Q$ which in addition belongs to~$\mathcal{Q}$. Furthermore,
we assume that for any rectangle $Q \in \mathcal{Q}$, its boundary satisfies
\begin{displaymath}
  \bdr{Q} = \bigcup \left\{ P \in \mathcal{Q} \; : \;
  P \mbox{ is a primary face of } Q \right\} \; .
\end{displaymath}
Finally, a {\em rectangular CW-complex\/}~$|\mathcal{Q}|$ is given by the
union of some rectangular structure~$\mathcal{Q}$, i.e., it is the subset
of Euclidean space which is occupied by the rectangles in~$\mathcal{Q}$.
Any $0$-dimensional rectangle in~$\mathcal{Q}$ is called a {\em vertex\/},
and if $Q \in \mathcal{Q}$ is an $n$-dimensional rectangle, then the set
$Q \setminus \bdr{Q}$ is called an {\em $n$-cell\/}. In order to
illustrate this definition further, we note that the rectangles in
a rectangular structure~$\mathcal{Q}$ satisfy the following two properties.
\begin{itemize}
\item Let~$\mathcal{Q}^{(k)}$ denote the rectangles in~$\mathcal{Q}$
of dimension~$k$. Then for~$k \in \mathbb{N}_0$ and
rectangles~$P,Q \in \mathcal{Q}^{(k)}$ we either have $P \cap Q = \emptyset$,
or the intersection~$P \cap Q$ is a face of both~$P$ and~$Q$.
\item For $k \in \mathbb{N}$ and $P \in \mathcal{Q}^{(k)}$, the union of
all rectangles in~$\mathcal{Q}^{(k-1)}$ which are primary faces of~$P$ equals
the boundary~$\bdr{P}$.
\end{itemize}
\begin{figure}
  \centering
  \includegraphics[width=15.5cm]{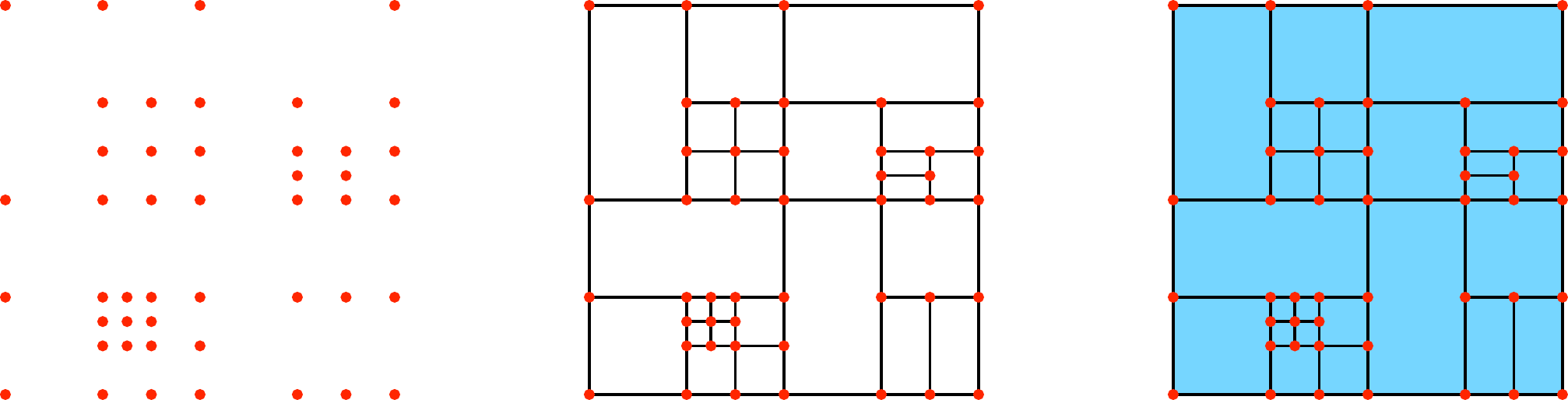}%
  \caption{An example of a rectangular CW-complex. The first image 
           shows the initial collection~$\mathcal{Q}^{(0)}$ of~$46$ 
           zero-dimensional rectangles, the middle image depicts
           the~$70$ one-dimensional rectangles contained in~$\mathcal{Q}^{(1)}$,
           and the last image shows the~$25$ two-dimensional rectangles
           in~$\mathcal{Q}^{(2)}$ which fill in the holes. The final rectangular
           CW-complex is a square.}
  \label{fig:rectCWcomplex}
\end{figure}
Notice that the above description can be used to give an algorithmic
description of building~$\mathcal{Q}$. Initially, one starts with a
collection of zero-dimensional rectangles, i.e., points, in~$\R^d$.
Then one adds one-dimensional rectangles connecting some of the points,
adds two-dimensional rectangles filling holes, etc. For an example
of a rectangular CW-complex, see Figure~\ref{fig:rectCWcomplex}. The
depicted CW-complex decomposes a square into smaller rectangles
through selected dyadic subdivisions. In the remainder of this paper,
we will only consider rectangular CW-complexes of this type, even 
though the above definition is more general. This is due to the fact
that the algorithms presented below avoid the use of derivatives, which
would allow for different adaptive decompositions.
\subsection{Persistent Homology of Rectangular CW-Complexes}
\label{sec:PersistentHomology2}
We now turn our attention to the definition of homology. It was shown
in~\cite[Theorem~3.9]{cw} that every rectangular CW-complex is a CW-complex
in the classical sense, see for example~\cite{massey}. This implies that if
a rectangular CW-complex defines a polyhedron, then its homology is isomorphic
to its singular homology. In order to further define and then efficiently
compute the homology of such a rectangular CW-complex, we make use of the
notion of \emph{incidence coefficients} or \emph{incidence numbers}. In the
simplest case of homology with coefficients in~$\mathbb{Z}_2$, the incidence
number of two rectangles $Q,P \in \mathcal{Q}$ satisfies $\kappa(Q,P) = 1$
if and only if~$P$ is a primary face of~$Q$, and we define $\kappa(Q,P) = 0$
in all other cases. Even for the more involved case of homology with
coefficients in a field with at least three elements, the incidence numbers
of a rectangular CW-complex can still be computed efficiently, for more
details see~\cite[Theorem~3.10]{cw}.

Suppose now that~$X \subset \mathbb{R}^d$ denotes a rectangular CW-complex
with rectangular structure~$\mathcal{Q}$. In order to rigorously define
the \emph{homology} of~$X$ or~$\mathcal{Q}$, we begin by establishing the
notions of chains, cycles, and boundaries. Let~$\mathbb{F}$ denote an
arbitrary coefficient field. Then a \emph{$k$-chain}~$c$ is a formal sum
of the form
\begin{displaymath}
  c = \sum_{A \in \mathcal{Q}^{(k)}} \alpha_A \, A \; ,
  \quad\mbox{ where }\quad
  \alpha_A \in \mathbb{F}
  \quad\mbox{ and }\quad
  \mathcal{Q}^{(k)} = \left\{ Q \in \mathcal{Q} \; : \; \dim Q = k \right\}
  \; .
\end{displaymath}
The set of all $k$-chains forms an additive group, which is called the
\emph{$k$-th chain group} and denoted by~$C_k(\mathcal{Q})$. For a rectangle
$Q \in \mathcal{Q}^{(k)}$, the \emph{boundary}~$\partial Q$ of~$Q$ is the
$(k-1)$-chain given by
\begin{displaymath}
  \partial Q =
  \sum_{P \in \mathcal{Q}^{(k-1)}} \kappa(Q,P) P \; ,
\end{displaymath}
which in view of the above discussion of the incidence
numbers~$\kappa(Q,P)$ is in fact only a sum over the primary faces of~$Q$,
see also~\cite[Section~3.1]{cw}. Through linear extension, this defines
a homomorphism~$\partial : C_k(\mathcal{Q}) \to C_{k-1}(\mathcal{Q})$,
which is called the \emph{boundary operator}. One can show that the
boundary operator satisfies
\begin{displaymath}
  \partial^2 = 
  \partial \circ \partial = 0 \; : \;
  C_k(\mathcal{Q}) \to C_{k-2}(\mathcal{Q}) \; .
\end{displaymath}
A chain $c \in C_k(\mathcal{Q})$ is called a \emph{$k$-cycle}, if
$\partial c = 0$. The set of all $k$-dimensional cycles forms an
additive group which is denoted by~$Z_k(\mathcal{Q})$. A $k$-cycle~$d$
is called a \emph{$k$-boundary}, if there exists a chain $e \in
C_{k+1}(\mathcal{Q})$ such that $\partial e = d$ holds. The
set of all $k$-boundaries is abbreviated by~$B_k(\mathcal{Q})$, and
one can easily see that it is a subgroup of~$Z_k(\mathcal{Q})$. Finally,
the \emph{homology group} of~$\mathcal{Q}$ in dimension~$k$ is then the
quotient group $H_k(\mathcal{Q}) = Z_k(\mathcal{Q})/B_k(\mathcal{Q})$.
From a geometric point of view, the $k$-th homology group~$H_k(\mathcal{Q})$
measures the number of $k$-dimensional holes in~$\mathcal{Q}$, and we
refer the reader to~\cite{herbert, massey} for more details.

As our last concept of this section, we now consider persistent homology.
While we keep our discussion brief, a more extensive introduction can be
found in~\cite{dwphysicad}. Consider a rectangular CW-complex given by the
rectangular structure~$\mathcal{Q}$. Then a subcollection~$\mathcal{Q}_0$
of~$\mathcal{Q}$ is called a \emph{subcomplex} of~$\mathcal{Q}$, if the
set~$\mathcal{Q}_0$ defines a rectangular structure in its own right.
In addition, a sequence of subcomplexes $\emptyset \subset \mathcal{Q}_0
\subset \ldots \subset \mathcal{Q}_{n-1} \subset \mathcal{Q}_n = \mathcal{Q}$
is called a \emph{filtration} of the rectangular CW-complex given
by~$\mathcal{Q}$. Filtrations lie at the heart of persistent homology,
which basically tracks topological changes in the CW-complexes of the
filtration as the index increases. While filtrations of CW-complexes can
arise in a number of ways, the following situation will be central for the 
present paper. Consider a function $f : \mathcal{Q} \rightarrow \mathbb{R}$
which assigns a real value to every rectangle in~$\mathcal{Q}$. Moreover,
assume that~$f$ is increasing in the sense that $f(P) \leq f(Q)$ whenever
$P \in \mathcal{Q}$ is a face of $Q \in \mathcal{Q}$. One can easily see
that if~$a_0 \le a_1 \le \ldots \le a_n$ are the finitely many real numbers
in the range of~$f$, then the sets
\begin{displaymath}
  \mathcal{Q}_k = \left\{ Q \in \mathcal{Q} \; : \;
  f(Q) \le a_k \right\}
\end{displaymath}
are rectangular structures, and therefore $\mathcal{Q}_0 \subset \ldots
\subset \mathcal{Q}_{n-1} \subset \mathcal{Q}_n = \mathcal{Q}$ is a filtration
of the rectangular CW-complex given by~$\mathcal{Q}$. We refer to this setting
as a \emph{filtered rectangular CW-complex}~$|\mathcal{Q}|$ given
by~$(\mathcal{Q},f)$.

Consider now an arbitrary filtration $\mathcal{Q}_0 \subset \ldots \subset
\mathcal{Q}_{n-1} \subset \mathcal{Q}_n$ of rectangular CW-complexes. Then
for $0 \le a \le b \le n$, we have $\mathcal{Q}_a \subset \mathcal{Q}_b$,
which in turn implies $C_k(\mathcal{Q}_a) \subset C_k(\mathcal{Q}_b)$ for
all $k \in \mathbb{N}_0$. This inclusion of the chain groups can be used
to define homomorphisms of the associated homology groups. More precisely,
we define the mapping $\iota^k_{a,b} : H_k(\mathcal{Q}_a) \rightarrow
H_k(\mathcal{Q}_b)$ by mapping any representative $k$-chain~$c$ of a homology
class in~$H_k(\mathcal{Q}_a)$ to the homology class generated by~$c$
in~$H_k(\mathcal{Q}_b)$. One can show that this is well-defined, and that
the mappings~$\iota^k_{a,b}$ are group homomorphisms. Then the \emph{$k$-th
persistent homology groups} are defined as
\begin{displaymath}
  H^{a,b}_k(\mathcal{Q}) = \Ima \iota^k_{a,b} \; ,
\end{displaymath}
i.e., as the images of the homomorphisms~$\iota^k_{a,b}$.

As mentioned above, persistent homology tracks topological changes through the
filtration, and this leads to the following notions for homology classes. We
say that a homology class $\gamma \in H_k(\mathcal{Q}_\ell)$ is \emph{born}
in~$H_k(\mathcal{Q}_\ell)$, if $\gamma \not \in \Ima \iota^k_{\ell-1,\ell}$.
The homology class~$\gamma$ \emph{dies} entering~$H_k(\mathcal{Q}_m)$, if
either it became trivial, or if it merges with an older class
in~$\mathcal{Q}_m$. In this context, a class is called \emph{older},
if it was born earlier than the homology class~$\gamma$. The index values~$\ell$
and~$m$ are called the \emph{birth time} and \emph{death time} of the homology
class~$\gamma$, respectively. We refer the reader to~\cite{herbert} for more
details.
\begin{figure}[tb]
  \centering
  \includegraphics[height=6cm]{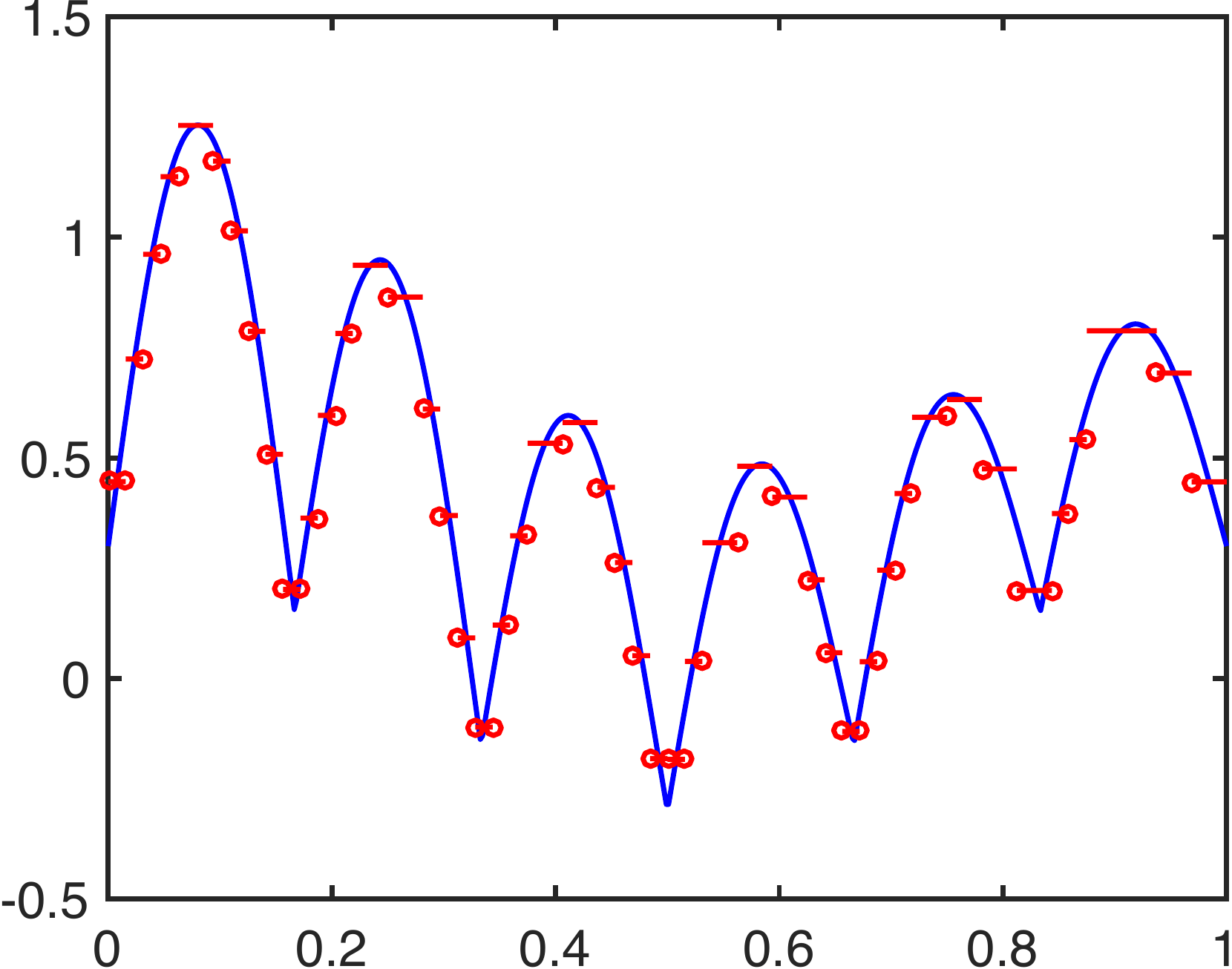}
  \hspace{0.5cm}
  \includegraphics[height=6cm]{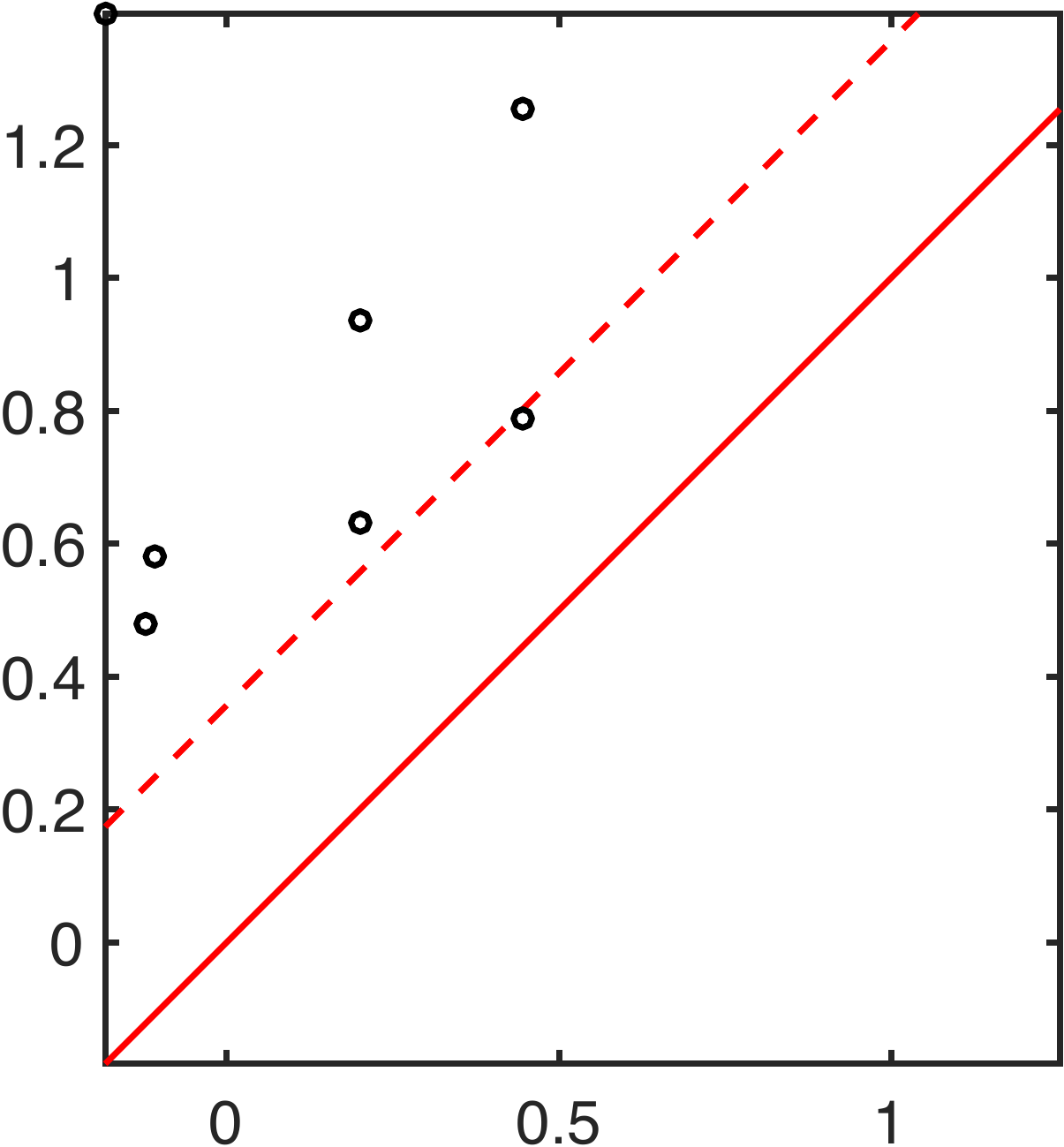} \\[4ex]
  \includegraphics[height=6cm]{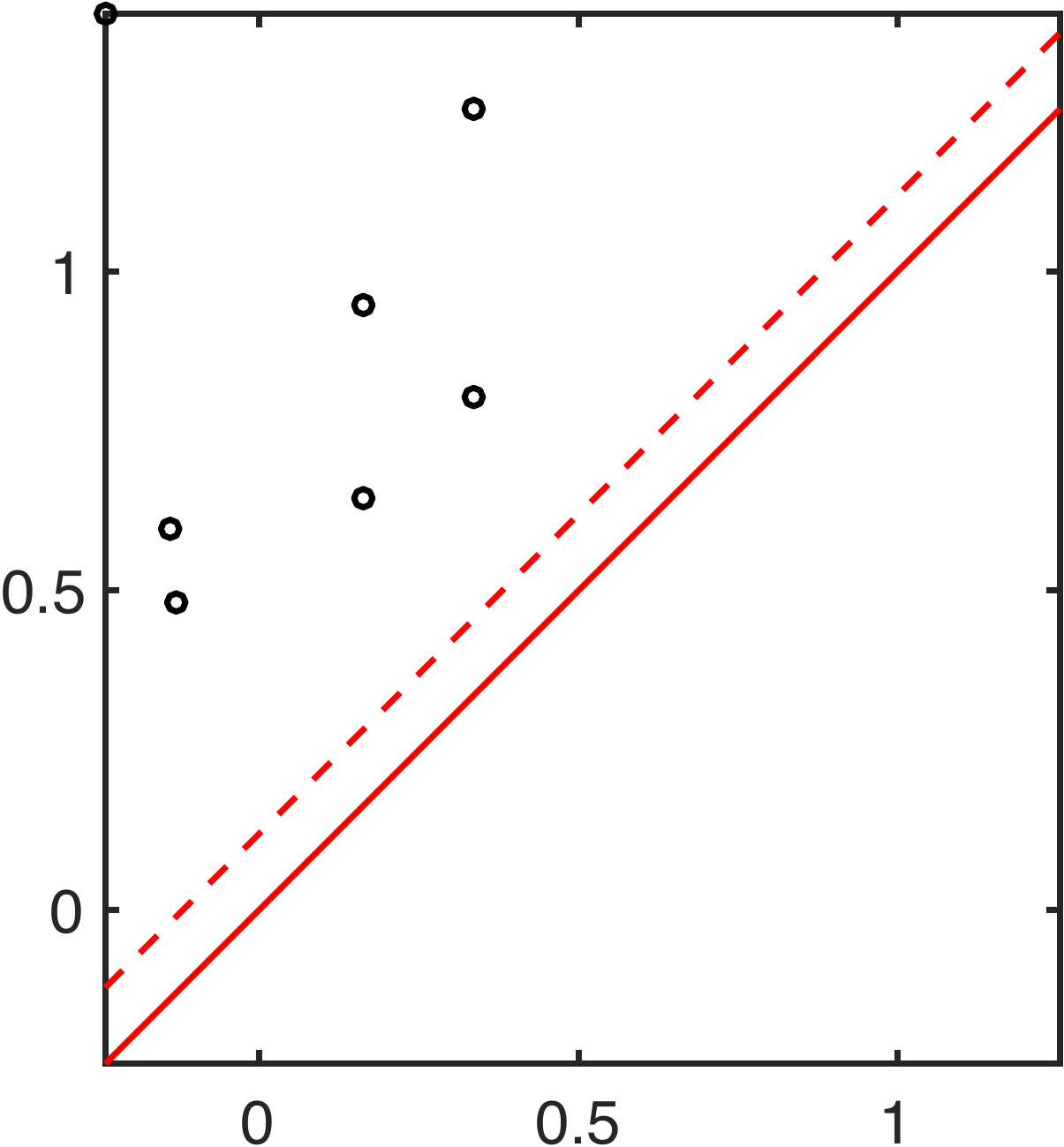}
  \hspace{0.5cm}
  \includegraphics[height=6cm]{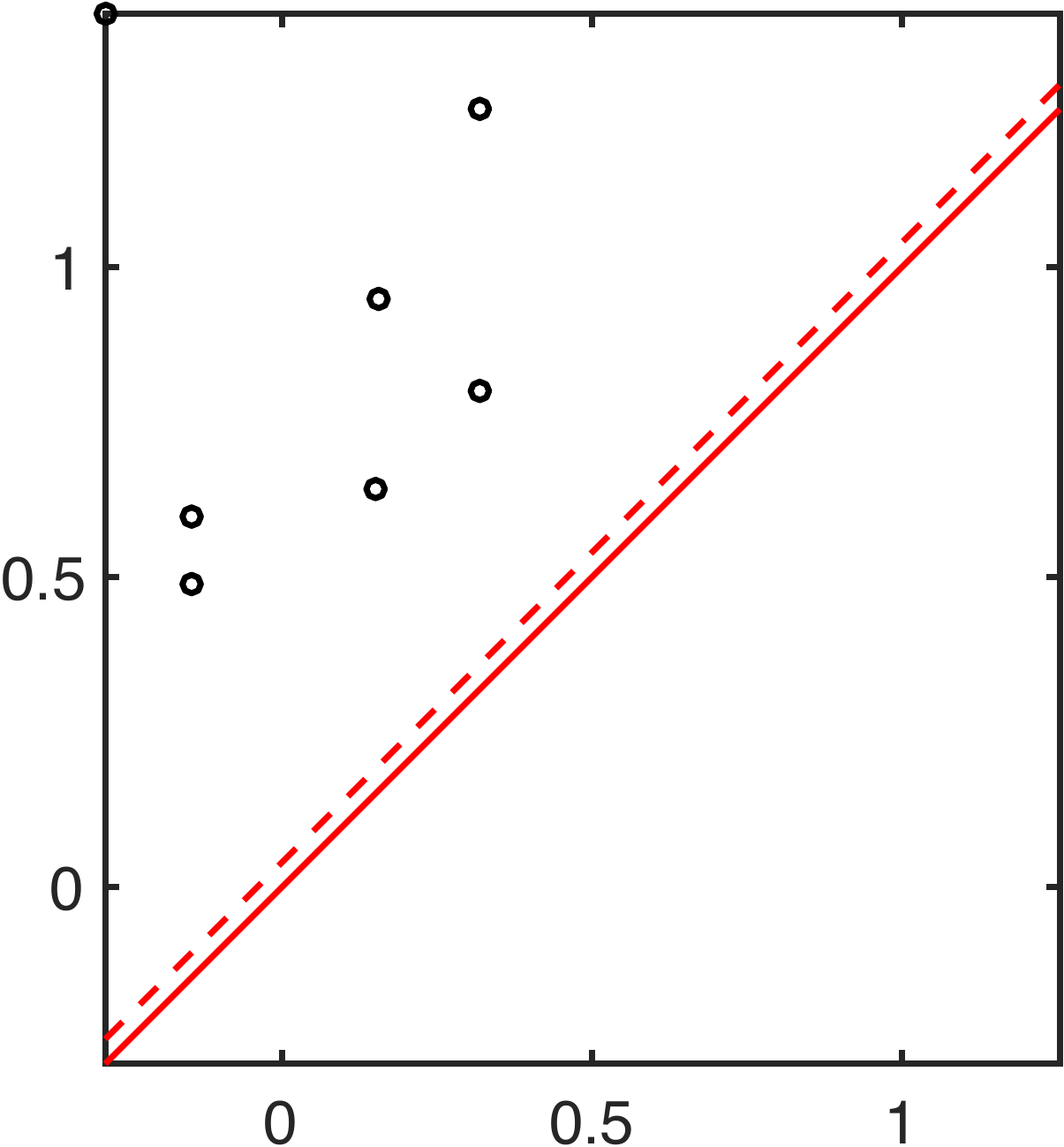}
  \caption{Sample persistence diagrams in dimension zero for piecewise
           constant approximations of the non-differentiable function~$f(x) =
           |\sin(6 \pi x)| / (1 + x^2) + 3 \cos(2 \pi x) / 10$. The
           function~$f$, together with its approximation~$\Box f$
           for~$\epsilon = 0.18$, is shown in the upper left image,
           while the upper right image contains the persistence
           diagram for~$\Box f$. The two lower panels show the
           persistence diagrams for~$\Box f$ if we use instead
           $\epsilon = 0.06$ (left) or $\epsilon = 0.02$ (right).}
  \label{fig:cornerex}
\end{figure}

The collection of all persistent homology groups is usually referred
to as the \emph{persistence module}. Despite its abstract nature, it can
be completely encoded via the notions of birth and death of homology
classes. For every homology class~$\gamma$, we record its birth time 
and its death time as an ordered pair, where the death time is set
to~$\infty$ if the class~$\gamma$ is nontrivial in~$\mathcal{Q}_n$.
In this way, a multiset of pairs of numbers~$\{(a_i,b_i)\}_{i \in I}$
is created, where $-\infty < a_i \leq b_i \le \infty$. This multiset
of points can naturally be drawn in the first quadrant of the Euclidean
plane, and in fact it lies above the line $b = a$. Such a representation
of the persistence module is called a \emph{persistence diagram}. If 
the persistence module is for a filtered rectangular
CW-complex~$X = |\mathcal{Q}|$ given by~$(\mathcal{Q},f)$, then we
denote the persistence diagram by~$D(f,X)$, or shorter by~$D(f)$ if the 
underlying space~$X$ is clear from context. Examples of persistence
diagrams can be found in the Figure~\ref{fig:cornerex}. In the top
left image, a continuous function~$f$ is shown together with its 
piecewise constant approximation~$\Box f$. Now let~$\mathcal{Q}^{(0)}$
denote the endpoints of the interval~$[0,1]$, together with the
collection of all points of discontinuities of~$\Box f$. If we denote
the points in the set~$\mathcal{Q}^{(0)}$ by $0 = x_0 < \ldots < x_N = 1$
and let $\mathcal{Q}^{(1)} = \{ [x_{k-1}, x_k] : k = 1,\ldots,N \}$,
then $\mathcal{Q} = \mathcal{Q}^{(0)} \cup \mathcal{Q}^{(1)}$ is a
rectangular structure, which renders the interval~$X = [0,1]$ a
filtered rectangular CW-complex given by~$(\mathcal{Q},\Box f)$.
The associated persistence diagram is shown in the upper right
panel. The two panels below show persistence diagrams for two
more approximations of~$f$, which will be discussed more later.
\subsection{Stability of Persistent Homology}
\label{sec:PersistentHomology3}
One of the prominent features of persistent homology is the fact that
the collection of persistence diagrams can be turned into a metric 
space in various ways, and we will present two of those below. In 
order to set the stage, we have to slightly modify the notion of 
persistence diagram. So far, a persistence diagram is a finite
multiset of points~$(a,b)$ with $a < b$, which represent birth/death-time
pairs for homology generators. For reasons that will become clearer
momentarily, we now add infinitely many of copies of all points~$(a,a)$
for $a \in \mathbb{R}$ to every persistence diagram, i.e., in addition
to the birth/death pairs every persistence diagram contains infinitely
many copies of the diagonal $b = a$. This modification ensures that any
two persistence diagrams~$A$ and~$B$ have the same infinite cardinality,
and there always are bijections $\eta : A \rightarrow B$. Notice that
without the addition of the diagonal, such bijections would generally
not exist. Now we can define the two central notions of distance
between~$A$ and~$B$.
\begin{itemize}
\item The \emph{bottleneck} distance between two persistence
diagrams~$A$ and~$B$ is defined as
\begin{displaymath}
  W_{\infty}(A,B) =
  \inf\left\{ \sup_{p \in A} \norm{p-\eta(p)}_{\infty} \; : \;
    \eta : A \rightarrow B \mbox{ is a bijection } \right\} \; .
\end{displaymath}
\item The \emph{degree-$q$ Wasserstein distance} between two persistence
diagrams~$A$ and~$B$ is defined as
\begin{displaymath}
  W_q(A,B) =
  \inf \left\{ \left( \sum_{p \in A} \norm{p-\eta(p)}_{\infty}^q
    \right)^{1/q} \; : \;
    \eta : A \rightarrow B \mbox{ is a bijection } \right\} \; .
\end{displaymath}
\end{itemize}
\begin{figure}
  \centering
  \includegraphics[width=4cm]{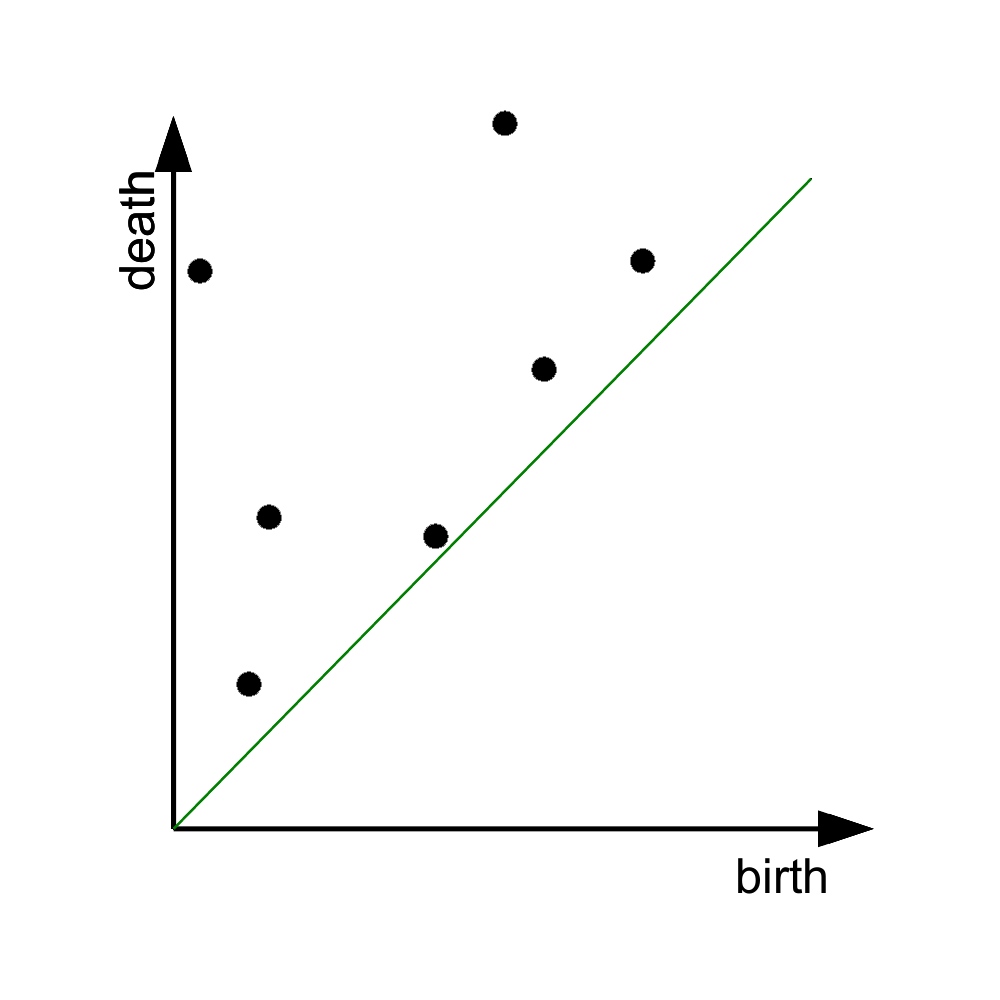}%
  \hspace{1.0cm}
  \includegraphics[width=4cm]{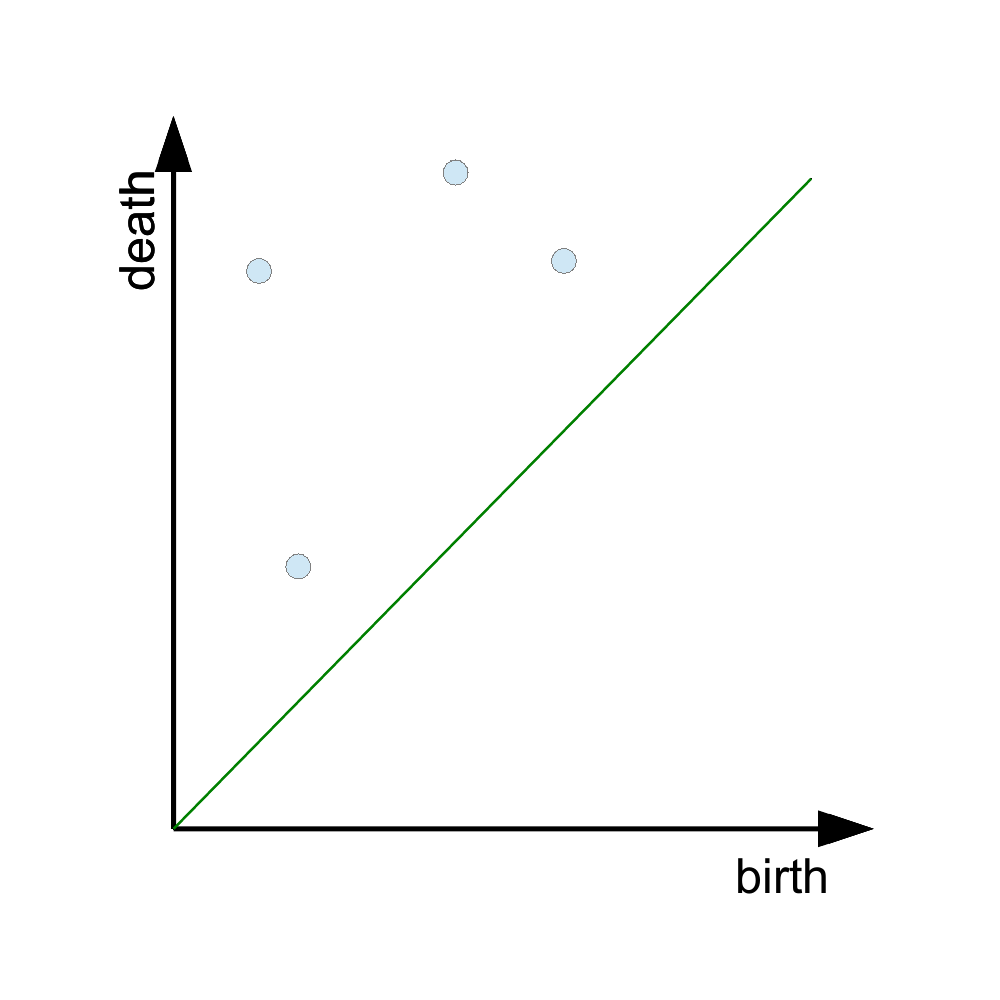}%
  \hspace{2.0cm}
  \includegraphics[width=4cm]{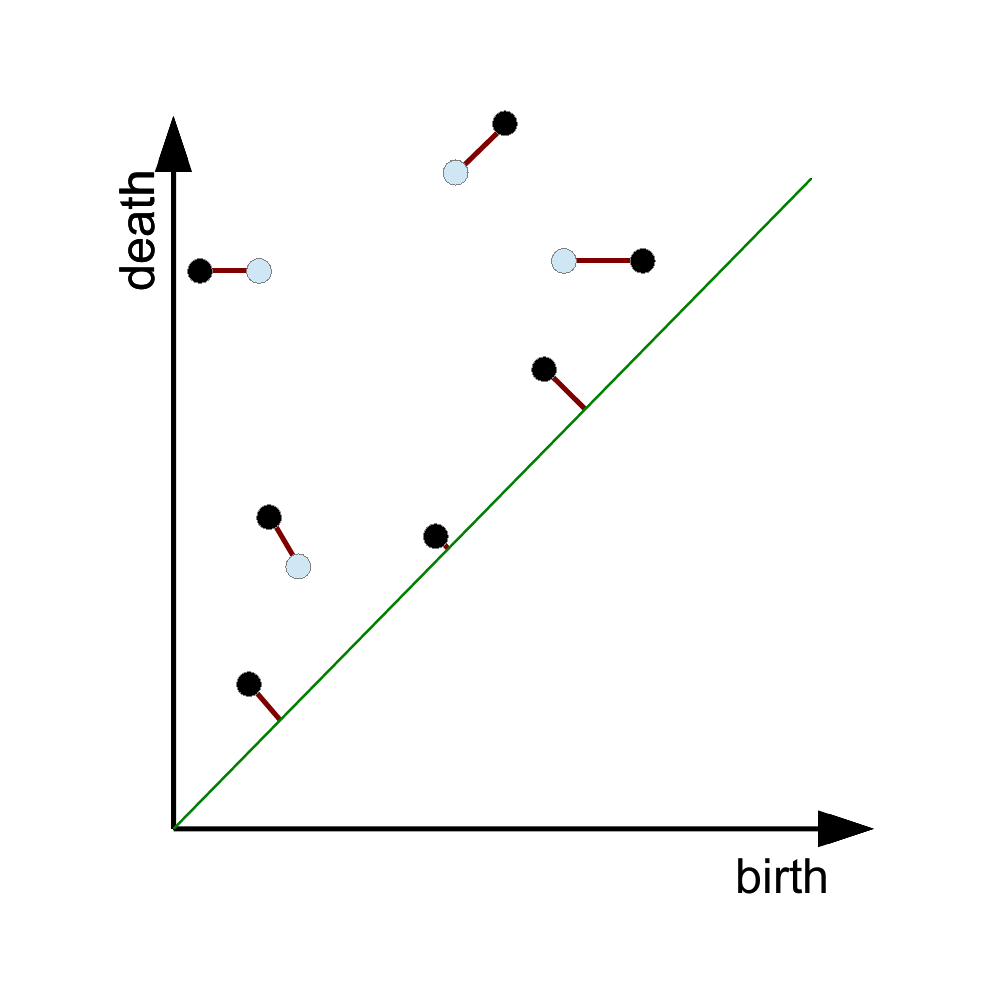}%
  \caption{Sample persistence diagram matching. The left two images show
           persistence diagrams with seven and four birth/death pairs,
           respectively. The third diagram shows a sample transport
           plan~$\eta$ between the two persistence diagrams.}
  \label{fig:bottelneckDistance}
\end{figure}
In other words, both distances try to identify a \emph{transport plan}
$\eta : A \rightarrow B$ which transforms the persistence diagram~$A$
into the persistence diagram~$B$, with in some sense minimal movement
of the points in the diagrams. While the bottleneck distance only considers
the largest displacement, the Wasserstein distance takes an average over all
displacements. See also Figure~\ref{fig:bottelneckDistance}.

Now that we are able to measure changes in persistence diagrams, one can
try to assess how changes in a filtration function~$f$ affect the associated
persistence diagrams. In fact, it is possible to establish a wide-ranging
stability property for persistent homology in this context, which is one
of the central reasons for its usefulness in the applied sciences.
Consider a function $f : X \rightarrow \mathbb{R}$ which is defined on a
triangulable topological space~$X$. One can define persistent homology
for the infinite filtration created by sublevel sets of~$f$, as long as 
the function~$f$ satisfies an additional property. To describe this in
slightly more detail, let~$X_a = f^{-1}(-\infty,a]$ for all $a \in \mathbb{R}$.
Then we clearly have the inclusion $X_a \subset X_b$ whenever $a \le b$.
As before, one would like to define the persistent homology groups as images
of the maps on homology induced by the inclusions $X_a \hookrightarrow X_b$.
Since we are faced with a filtration with infinite index set, certain
finiteness assumptions have to be made. For this, one calls a number
$a \in \mathbb{R}$ a \emph{homological critical value}, if there is
no value $\epsilon > 0$ such that the inclusion $X_{a-\epsilon}
\hookrightarrow X_{a+\epsilon}$ induces only isomorphisms on the homology
level. In other words, at a homological critical value the homology
groups change. We then call the function~$f$ \emph{tame}, if it only
has finitely many homological critical values, and if all homology
groups of all sublevel sets~$X_a$ have finite rank. For tame functions,
it is possible to define persistence diagrams similar to the procedure
described above, see for example~\cite[pp.~182f]{herbert}. Moreover,
the following stability result was established by Cohen-Steiner,
Edelsbrunner, and Harer~\cite{cohensteiner:etal:07a}.
\begin{theorem} \label{th:stabilityTheoremOrig}
Let~$X$ be a triangulable space, and let $f,g : X \rightarrow \mathbb{R}$
denote two continuous tame functions. Then the persistence diagrams~$D(f)$
and~$D(g)$ of~$f$ and~$g$, respectively, satisfy
\begin{displaymath}
  W_{\infty}(D(f),D(g)) \leq \norm{f-g}_{\infty} \; .
\end{displaymath}
\end{theorem}
While the above theorem was a milestone for the applicability of 
persistent homology, it did suffer from a serious drawback. In its
original formulation, the functions~$f$ and~$g$ have to be at least
continuous. This of course would preclude our idea of approximating
a continuous functions via a piecewise constant one. Fortunately,
this continuity restriction has been removed in the work by
Chazal et al.~\cite{chazal:etal:09a, chazal:etal:12a}, which is
based on interleavings of persistence modules. As a consequence, 
we have the following result.
\begin{theorem} \label{th:stabilityTheorem}
Let~$X$ be a rectangular CW-complex, and let $f,g : X \rightarrow \mathbb{R}$
denote two functions, each of which is either continuous and tame, or piecewise
constant on the rectangles of the rectangular structure associated with~$X$
and lower semi-continuous. Then the persistence diagrams~$D(f)$ and~$D(g)$
of~$f$ and~$g$, respectively, satisfy
\begin{displaymath}
  W_{\infty}(D(f),D(g)) \leq \norm{f-g}_{\infty} \; .
\end{displaymath}
In other words, small perturbations in the function~$f$ only lead to
small perturbations of the persistence diagram~$D(f)$ with respect to
the bottleneck distance.
\end{theorem}
One of the crucial properties of persistent homology is its inherent
computability and a number of software packages have been developed
for this, for example {\tt dionysus}~\cite{dionysus},
{\tt PHAT}~\cite{phat}, {\tt Perseus}~\cite{perseus} and {\tt Gudhi}~\cite{gudhi}. While two
of these are based on standard matrix reductions~\cite{phat, dionysus},
one is based on discrete Morse theory~\cite{perseus}, see
also~\cite{MichMorse}. All of the packages are designed to compute
the persistence of a finite cell complex and are therefore not
capable of solving the problem of computing persistence of a
continuous function of a compact set to within a pre-specified
error bound. This gap in the literature is filled by the present
paper.

Persistent homology can also be defined for vector-valued filtering
functions $f : X \rightarrow \mathbb{R}^n$, where~$X$ is a rectangular
CW-complex or a triangulable space, and this leads to the concept of
\emph{multi-dimensional persistence}. However, for multi-dimensional
persistence one can show that there is no finite and complete
invariant~\cite{gunnar} which is similar to the persistence intervals
in the one-dimensional case $f : X \rightarrow \mathbb{R}$. Needless
to say, this implies that there is no hope for an algorithm which 
computes multi-dimensional persistence. To address this issue, a number
of research groups have tried to find some descriptors of multi-dimensional
persistence, which --- though incomplete -- can still prove to be useful
in practice and which are computable. It seems likely that in order to
compute such invariants it is necessary to construct a finite cell complex.
In fact, our algorithms presented below will work for any continuous
vector-valued function which is defined on a rectangular domain, and they
produce a piecewise constant lower semi-continuous approximation of the
function. Given such an approximation, one can use various stability
results for multi-dimensional persistence to show that the
multi-dimensional persistence information of the approximation
is close to the multi-dimensional persistence of the given continuous
function.

To close this section, we present one such stability result
for multi-dimensional persistence which is taken from the
paper~\cite{stabilityPersistentBettiNumbers}. Their construction
is based on finding a matching distance between rank invariants
of multi-dimensional persistent homology. For this, the following
concepts are introduced in~\cite{stabilityPersistentBettiNumbers}.
Let $u, v \in \mathbb{R}^n$ be given, and write $u = (u_1,\ldots,u_n)$
and $v = (v_1,\ldots,v_n)$. We say that $u \preceq v$ if for every
$i \in \{1,\ldots,n\}$ we have $u_i \leq v_i$. Define
the set $\delta^{+} = \{ (u,v) \in \mathbb{R}^n \times \mathbb{R}^n
\; : \; u \preceq  v \}$, and consider the usual maximum
norm~$\| \cdot \|_\infty$ on~$\mathbb{R}^n$. For a fixed triangulable
space~$X$ and a function $f : X \rightarrow \mathbb{R}^n$, define the
sublevel set $X(f,u) = \{ x \in X \; : \; f(x) \preceq u \}$. Now
consider two vectors~$u \preceq v$, and let $\pi_k^{u,v} : H_k(X(f,u))
\rightarrow H_k(X(f,v))$ denote the homomorphism induced by inclusion.
Then the image of~$\pi_k^{u,v}$ is called the \emph{multi-dimensional
$k$-th persistent homology group} of~$(X,f)$ at~$(u,v)$, where the
authors of~\cite{stabilityPersistentBettiNumbers} use \emph{\v{C}ech
homology} over a field, see~\cite{massey} for more details.
Additionally, the \emph{$k$-th rank invariant} of a pair~$(X,f)$ over a
fixed coefficient field~$K$ is defined as the function $\rho(X,f,k) :
\delta^{+} \rightarrow \mathbb{N} \cup \{ \infty \}$ given by
$\rho(X,f,k)(u,v) = {\rm rank} \, \pi_k^{(u,v)}$. For the so-defined
rank invariant, the following stability theorem was established
in~\cite[Theorem~4.4]{stabilityPersistentBettiNumbers}.
\begin{theorem} \label{th:foliationPaper}
Let~$X$ be a triangulable space, let $f,g : X \rightarrow \mathbb{R}^n$
be continuous, and let $k \in \mathbb{Z}$ be arbitrary. Then  there
exists a distance function~$D_{match}$ on the collection of all possible
rank invariant functions such that
\begin{displaymath}
  D_{match}\left( \rho(X,f,k), \rho(X,g,k) \right) \; \le \;
  \max_{x \in X}\|f(x)-g(x)\|_{\infty} \; .
\end{displaymath}
\end{theorem}
As before, Theorem~\ref{th:foliationPaper} provides a bound on the change
of the rank invariant of multi-dimensional persistent homology in terms of
the change of the underlying function~$f$ measured in the maximum norm. We
would like to point out that the above results is just one example of a number
of similar results. In the same paper~\cite{stabilityPersistentBettiNumbers},
the authors derive a similar stability result for \emph{persistent Betti
numbers}, and in~\cite{hausdofrrStability} the case of \emph{multi-dimensional
persistence spaces} was considered. In all of these examples, the right-hand
side of the estimate is of the form presented in Theorem~\ref{th:foliationPaper}.
It is therefore clear that bottleneck-type stability can be observed in
multi-dimensional persistence, and the discretization process presented
below may be useful for computations of some of the above invariants.
\section{Verified Cubical Approximation of Continuous Functions}
\label{sec:cubapprox}
In this section, we present the main algorithms for approximating
continuous functions by piecewise constant lower semi-continuous
functions. Our approach makes use of rigorous computations via 
interval arithmetic, and will allow for the approximation of
the persistent homology of a function with rigorous error bounds.
\subsection{Rigorous Numerics and Interval Arithmetic}
\label{sec:intervalArithmetic}
We begin by collecting a few definitions and results from interval
arithmetic, which are necessary for our approach. In what follows, we will
use interval arithmetic to obtain rigorous mathematical statements about
approximations of continuous functions. We would like to point out, however,
that alternative rigorous computer arithmetic methods could be used just as
well, see for example~\cite{affineArithmetic, generalizedIntervalArithmetic}.
For a comprehensive introduction to interval arithmetic we refer the reader
to~\cite{moore}.

Throughout this paper our computations work with a fixed set~$\mathcal{R}$
of real numbers which are representable in a computer, called the
\emph{representable numbers}~$\mathcal{R}$. In order to account for
rounding errors, rather than working with numbers we make use of intervals.
This leads to the concept of \emph{interval arithmetic}, which represents
every real number in~$\mathbb{R}$ as an interval with endpoints in~$\mathcal{R}$.
More precisely, if~$x \in \mathbb{R}$ is a representable number, it is
represented by the degenerate interval~$[x,x]$, while any number $x \in
\mathbb{R} \setminus \mathcal{R}$ is represented as~$[\underline{x},
\overline{x}]$ with~$\underline{x}$ being the largest representable
number which is smaller than~$x$, and similarly for~$\overline{x}$.

In order to work with functions, one has to define arithmetic operations
on intervals. This has to be done in such a way that any potential rounding
errors are taken into account. To describe this in more detail,
let~$\mathcal{I}(\mathcal{R})$ denote the set of all intervals with both
endpoints in the set of representable numbers. Given two intervals
$a, b \in \mathcal{I}(\mathcal{R})$ with $a = [a_1,a_2]$ and $b = [b_1,b_2]$,
and denoting any of the four basic arithmetic operations by~$\diamond$,
the analogous interval operation $a \diamond b$ has to be defined in such
a way that
\begin{displaymath}
  \left\{ x \diamond y \; : \; x,y \in \mathbb{R} \text{ and }
    x \in a, \; y \in b \right\}
  \; \subset \;
  a \diamond b
  \; \in \;
  \mathcal{I}(\mathcal{R})
  \; .
\end{displaymath}
In practice, this is achieved in computer implementations of interval
arithmetic by outward rounding to the next representable numbers. In this
way, the resulting interval~$a \diamond b$ always contains all possible
outcomes of the actual operation. In addition to the four basic arithmetic
operations, implementations of interval arithmetic also have to provide 
interval versions of the standard elementary functions. This is usually
accomplished via Taylor series expansions, where the resulting interval
answer will also include any potential truncation errors. For a more
comprehensive treatment of these issues we refer the read to~\cite{neumaier}.
Common to all practical computer implementations of interval arithmetic is
that they provide a collection of predefined continuous functions, which
usually at least contains functions in $\Phi = \{ {\rm abs}, \;
{\rm sqrt}, \; {\rm sqr}, \; {\rm exp}, \; {\rm ln}, \; {\rm sin}, \;
{\rm cos}, \; {\rm arctan} \}$.

For bounds on the complexity of our algorithms, we need to assume later
on that the functions of interest are at least Lipschitz continuous.
Recall that a function $f : D \subset \mathbb{R}^n \rightarrow \mathbb{R}^m$
is called \emph{Lipschitz continuous} on the domain $D_0 \subset D$ if there
exists a constant $L \ge 0$ such that
\begin{displaymath}
  \|f(x_1) - f(x_2)\|_{\infty} \leq L \|x_1 - x_2\|_{\infty}
  \quad\mbox{ for all }\quad
  x_1,x_2 \in D_0 \; .
\end{displaymath}
Clearly, every elementary function from the set~$\Phi$ is Lipschitz
continuous, at least on compact rectangles over which the function is
differentiable. In fact, the same statement is true for any
\emph{arithmetical expression} which is formed using functions
from~$\Phi$, the four basic arithmetic operations, and composition
of functions. For a more formal definition,
see~\cite[Section~1.4]{neumaier}.

Despite the fact that arithmetical expressions usually lead to Lipschitz
continuous functions, this does not immediately carry over to our machine
computations. In order for our results below to hold, we need to make sure
that if~$f$ is a Lipschitz continuous function, then its actual interval
implementation~$\hat{f}$ as a function from a subset of~$\mathcal{R}^n$
into~$\mathcal{R}^m$ is Lipschitz continuous as well. This is the subject
of the following result from~\cite[Theorem~2.1.1]{neumaier}.
\begin{theorem}[Arithmetical Expressions Inherit Lipschitz Continuity]
\label{th:neumaier}
Let~$f$ denote an arithmetical expression in~$n$ variables, which takes
$m$-dimensional values, i.e., suppose $f : D \subset \mathbb{R}^n
\rightarrow \mathbb{R}^m$, and suppose that~$f$ is Lipschitz continuous
on~$D_0 \subset D$ with constant~$L$. Furthermore, let~$\hat{f}$ denote
its interval implementation as described above, and let~$\mathcal{D}_0$
denote the interval vectors in~$\mathcal{I}(\mathcal{R})^n$ which
represent the points in~$D_0$. Then~$\hat{f}$ is Lipschitz continuous
on~$\mathcal{D}_0$. More precisely, we have
\begin{displaymath}
  \left\| \hat{f}(x_1) - \hat{f}(x_2) \right\|_{\infty} \leq
    \alpha(L) \left\| x_1 - x_2 \right\|_{\infty}
  \quad\mbox{ for all }\quad
  x_1, x_2 \in \mathcal{D}_0 \; ,
\end{displaymath}
where the constant~$\alpha(L)$ depends on the arithmetical expression
for~$f$, and can be computed explicitly using~\cite[Table~2.1]{neumaier}.
\end{theorem}
The above Theorem~\ref{th:neumaier} implies that interval evaluations
of a Lipschitz continuous function determine again a Lipschitz continuous
function. Of course, one would expect that the Lipschitz constants of the
new function~$\hat{f}$ would increase in size, as the implementation
deals with intervals. But as shown in~\cite[Table~2.1]{neumaier}, this
effect is minor.
\subsection{Piecewise Constant Approximations}
\label{sec:approximationContFunctions}
In this section we show how a continuous function $f : R \rightarrow
\mathbb{R}^m$, where $R \subset \mathbb{R}^n$ is a rectangular domain, 
can be approximated by a piecewise constant function. The approximation
will be constructed as a collection of compact rectangles whose union
is~$R$, where each of the rectangles in the decomposition is carrying
the following information:
\begin{itemize}
\item Each rectangle is described by an $n$-vector of pairs of
representable numbers $\{(b_i,e_i)\}_{i=1}^{n}$ which satisfy
the inequalities $b_i \leq e_i$, and which are the projections
of the rectangle onto the coordinate axes $x_1,\ldots,x_n$.
\item Each rectangle has an associated \emph{value}, which is an
$m$-dimensional vector of representable numbers. For a rectangle~$Q$
in the final decomposition of~$R$, its value is denoted by~$val(R)$. 
\end{itemize}
The algorithm which leads to the desired decomposition of~$R$ is
motivated by the classical simplicial approximation theorem. In its
original form, this theorem states that every continuous function
between triangulable spaces can be approximated by a simplicial map.
In order to obtain this simplicial map, one has to subdivide the
simplices, possibly many times. As far as we are aware, there is no criterion that allows to check if the approximation of a desired accuracy have been achieved. Saying so, the simplicial approximation theorem is non constructive. 

The constructive algorithm presented below uses
a similar idea. By employing rigorous arithmetic, for example interval
arithmetic as discussed in the previous section, we can ensure that
the final subdivision of the rectangular domain~$R$ is such that
the value~$val(Q)$ of every rectangle~$Q$ in the final subdivision
of~$R$ satisfies the estimate
\begin{displaymath}
  \left\| f(x) - val(Q) \right\|_\infty \; \leq \; \epsilon
  \quad\mbox{ for all }\quad
  x \in Q \; .
\end{displaymath}
To achieve this property, we repeatedly subdivide~$R$ until the above
estimate holds for every rectangle in the subdivision.
\begin{algorithm}[tb]
  \small
  \caption{Piecewise constant approximation of continuous functions.}
  \label{alg:partiallyConstantApproximationOfAFunction}
  \begin{algorithmic}
	\REQUIRE Compact rectangle $R \subset \mathbb{R}^n$,
	         continuous function $f : R \rightarrow \mathbb{R}^m$,
	         error bound $\epsilon > 0$.
    \ENSURE Rectangular partition of~$R$ which gives an $\epsilon$-approximation of~$f$ on~$R$.
    \STATE Queue $\mathcal{Q}$;
    \STATE $\mathcal{Q} \leftarrow R$;
    \WHILE {$\mathcal{Q}$ is not empty}
		\STATE Rectangle $Q \leftarrow pop(\mathcal{Q})$;
		\STATE $v = midpt(\hat{f}(Q))$;
		\IF { $(max(\hat{f}(Q)) - v \le \epsilon) \;\; {\rm and} \;\; (v - min(\hat{f}(Q)) \le \epsilon$) }
			\STATE $val(Q) = v$;
		\ELSE
			\STATE Subdivide $Q$ in all directions.
			       Put all resulting rectangles in~$\mathcal{Q}$;
		\ENDIF
    \ENDWHILE
    \RETURN 
  \end{algorithmic}
\end{algorithm}

To describe the algorithm in more detail, we need to introduce some
notation. Consider a rectangle $R \subset \mathbb{R}^n$. Then we say
that the rectangles $R_1,\ldots,R_k$ form a \emph{partition} of the
rectangle~$R$ if we have $R = R_1 \cup \ldots \cup R_k$, and if for
arbitrary $i,j \in \{1,\ldots,k\}$ the $n$-dimensional measure of the
intersection~$R_i \cap R_j$ is zero. The intersection $R_i \cap R_j$
can either be empty, or a rectangle of dimension smaller than~$n$.
A partition $R_1,\ldots,R_k$ is an \emph{$\epsilon$-approximation
of a function $f : R \subset \mathbb{R}^n \rightarrow \mathbb{R}^m$},
if we have
\begin{displaymath}
  \left\| f(x) - val(R_i) \right\|_\infty \; \leq \; \epsilon
  \quad\mbox{ for all }\quad
  x \in R_i
  \quad\mbox{ and }\quad
  i = 1,\ldots,k \; .
\end{displaymath}
With each $\epsilon$-approximation of~$f$, we can associate a
lower semi-continuous function~$\Box f$ which is defined via
$\Box f(x) = val(R_i)$ for all~$x$ in the interior of~$R_i$,
and for any point~$x$ which lies in the boundary of at least
one of the partition rectangles we define $\Box f(x) = \min\{
val(R_{i_1}), \ldots, val(R_{i_\ell}) \}$, where~$x$ is contained
in the rectangles $R_{i_1}, \ldots, R_{i_\ell}$. Finally, for a
rectangle $R = [a_1,b_1] \times \ldots \times [a_n,b_n]$ its
\emph{midpoint} is given by $midpt(R) = ((a_1+b_1)/2, \ldots,
(a_n+b_n)/2)$, and its geometric realization, i.e., the set of
points belonging to the rectangle, is denoted by~$|R|$. In terms
of the given function~$f$, we assume that it is of a form suitable
for evaluation over rectangles using the rigorous computer arithmetic.
With these definitions, the algorithm for determining a piecewise
constant approximation of a continuous function is given by
Algorithm~\ref{alg:partiallyConstantApproximationOfAFunction}.
In the algorithm, the rigorous arithmetic implementation of~$f$
is denoted by~$\hat{f}$.

The properties of rigorous computer arithmetic
discussed in the previous section readily imply that if
Algorithm~\ref{alg:partiallyConstantApproximationOfAFunction}
terminates, then the constructed partition is an
$\epsilon$-approximation of~$f$. Of course, there are
instances where the algorithm does not terminate. To see
this, suppose we are computing on a machine whose smallest
strictly positive representable number is $\rho_{\min} > 0$.
Choose $\delta > 0$ such that
\begin{equation} \label{deltachoice}
  \delta \left( \rho_{\min} + \delta \right) <
  \frac{\rho_{\min}}{2 \pi} \; ,
\end{equation}
and consider the function~$f : R \to \mathbb{R}$ defined
via
\begin{equation} \label{terminateex}
  f(x) = \sin\frac{1}{x + \delta}
  \qquad\mbox{ for all }\qquad
  x \in R = [0,1] \; .
\end{equation}
Any partition of~$[0,1]$ into compact intervals whose endpoints
are representable numbers has to contain an interval~$[0,a]$
with $\rho_{\min} \le a \le 1$. One can easily see
that~(\ref{deltachoice}) is equivalent to
\begin{displaymath}
  \frac{1}{0 + \delta} - \frac{1}{a + \delta}
  \; \ge \; 
  \frac{1}{0 + \delta} - \frac{1}{\rho_{\min} + \delta}
  \; > \;
  2 \pi \; ,
\end{displaymath}
and therefore the range of~$f$ over the interval~$[0,a]$ is
the interval~$[-1,1]$. This immediately implies that any
evaluation using rigorous computer arithmetic yields
$\hat{f}([0,a]) \supset [-1,1]$, and the algorithm
cannot terminate if we set $\epsilon < 1$.

The above example illustrates that even for simple analytic
real-valued functions on a compact one-dimensional domain
the presented algorithm may not terminate. Therefore, in order
to avoid infinite loops one has to set a maximal number of
subdivisions that the algorithm is allowed to perform. If this
number is exceeded, the algorithm terminate with the answer
``\emph{cannot decide},'' i.e., the algorithm is a \emph{partial
algorithm}. For clarity of presentation, this simple modification
is not part of the pseudo-code of
Algorithm~\ref{alg:partiallyConstantApproximationOfAFunction}.

Nevertheless, on a purely theoretical level, the algorithm
has the potential to terminate for every continuous function
$f : R \to \mathbb{R}^m$ defined on a compact rectangle
$R \subset \mathbb{R}^n$. This is a direct consequence of
the resulting uniform continuity of~$f$ on~$R$. In fact,
if we assume more regularity of~$f$, and that the computer
implementation has infinite precision, then it is even
possible to give an upper bound on the number of rectangles
in an $\epsilon$-approximation of~$f$. This is the subject
of the following result.
\begin{theorem}[Algorithm Complexity for Lipschitz Functions]
\label{thm:complexity}
Let $R \subset \mathbb{R}^n$ denote a compact rectangle,
and assume that the function $f : R \rightarrow \mathbb{R}^m$
is Lipschitz continuous with Lipschitz constant~$L$. Suppose
further that Algorithm~\ref{alg:partiallyConstantApproximationOfAFunction}
was successfully executed for an arithmetical expression of~$f$, for
the algorithm parameter $\epsilon > 0$, and that the algorithm was run
on a computer with infinite precision. Then the algorithm creates
a subdivision of the rectangle~$R$ with
\begin{displaymath}
  \mbox{at most }\quad
  \left( \frac{\alpha(L) diam(R)}{2\epsilon} \right)^n
  \quad\mbox{ rectangles,}
\end{displaymath}
where the constant~$\alpha(L)$ is taken from
Theorem~\ref{th:neumaier}, and $diam(R)$ denotes the diameter
of~$R$ with respect to the maximum norm.
\end{theorem}
\begin{proof}
In infinite precision, the {\tt if}-condition in
Algorithm~\ref{alg:partiallyConstantApproximationOfAFunction}
is satisfied as soon as we have
\begin{displaymath}
  max(\hat{f}(R)) - min(\hat{f}(R)) \; \le \;
  2 \epsilon \; ,
\end{displaymath}
where~$\hat{f}$ denotes the interval implementation of~$f$. In other
words, the algorithm terminates on a rectangle $A \subset R$, as soon
as the evaluation of~$\hat{f}$ has diameter at most~$2\epsilon$.
Moreover, Theorem~\ref{th:neumaier} implies that the interval
representation has the Lipschitz constant~$\alpha(L)$. Thus,
for $diam(A) \leq \frac{2\epsilon}{\alpha(L)}$ the rigorous
evaluation of~$f$ on~$A$ has diameter smaller at most~$2\epsilon$.
In the worst case, this means that the rectangle~$R$ has to
be subdivided $(\alpha(L) diam(R)) / (2\epsilon)$ times in each
coordinate direction, and this will lead to at most
$((\alpha(L) diam(R)) / (2\epsilon))^n$ rectangles in the
final subdivision.
\end{proof}
The above result was formulated for infinite precision in order
to exclude examples such as~(\ref{terminateex}). In actual computer
implementations, the number of rectangles in the final subdivision
of the rectangle is usually very close to the bound in the theorem,
but it might be slightly larger. This indicates that the techniques
which will be presented below can be expected to work extremely well
in low dimensions. On average, however, we expect the complexity of
the procedures developed in this paper to grow exponentially with the
dimension of the ambient space, in accordance with the above theorem.
Of course, the main factors in the complexity are the topological
features of the considered function, and there are many functions for
which this exponential growth is not of practical importance. Yet, as
soon as a function is Lipschitz-continuous with optimal Lipschitz constant
$L \geq 1$, the computational complexity grows exponentially with the
dimension. We would like to point out that this is yet another aspect
of the so-called \emph{curse of dimensionality}. When constructing
cubical complexes of an $n$-dimensional box, we are in fact experiencing
the same obstacles as the ones encountered in the construction of
Delaunay triangulations of Voronoi diagrams of~$k$ points
in~$\mathbb{R}^n$, which requires a computation time of the
order~$O(k^{\floor{(n+1)/2}})$, see for example~\cite{Chazelle}.
\subsection{The Subdivision Algorithm}
\label{sec:SubdivAlgorithm}
The algorithm presented in Section~\ref{sec:approximationContFunctions}
operates entirely on top-dimensional cells. This is sufficient if one's
only interest is the computation of an $\epsilon$-approximation of a
continuous function. If, however, the goal of the computation is the
(multi-dimensional) persistent homology of the underlying continuous
function, then it is necessary to create a complex of some sort, together
with an associated filtration. In our situation, this complex will be
a rectangular CW-complex, and in addition to the top-dimensional cells
one needs to efficiently store also the necessary lower-dimensional
cells. The creation of this complex will be described in the current
section.

In the following, we assume that the initial rectangle~$R$ used in
Algorithm~\ref{alg:partiallyConstantApproximationOfAFunction} is given
as a rectangular CW-complex. More precisely, we assume that it is stored
in a structure which contains all the boundary elements, and such that
each rectangle in the structure can access both its boundary and
coboundary rectangles. For this type of input, we describe a
straightforward subdivision scheme which allows one to subdivide~$R$
and incorporate the resulting new rectangles, while keeping the structure
of the rectangular CW-complex updated after each subdivision. In 
pseudo-code notation, this procedure can be implemented as in
Algorithm~\ref{alg:subdividingRectangleAlgorithm}. This algorithm
subdivides a rectangle~$R$ in the direction of the $i$-th coordinate
into two sub-rectangles~$R_1$ and~$R_2$ which have the same dimension
as~$R$, and which share a common face~$R_3$. The algorithm uses references
in~$R$ to the boundary and coboundary elements of~$R$ to maintain the global
structure of the rectangular CW-complex. See also
Figure~\ref{fig:rectangleDivisionExample} for an illustration. While
the above algorithm only subdivides the rectangle~$R$ along one coordinate
direction, in some cases one might like to perform subdivisions in all
directions at once. In this case, one can use
Algorithm~\ref{alg:subdividingRectangleAllDirectionsAlgorithm}.
\begin{algorithm}[tb]
  \small
  \caption{Rectangle subdivision in one coordinate direction.}
  \label{alg:subdividingRectangleAlgorithm}
  \begin{algorithmic}
	\REQUIRE $R = [r_{1,0},r_{1,1}] \times \ldots \times [r_{n,0},r_{n,1}]$, \;
	  $i \in \{1,\ldots,n\}$, \; $x \in (r_{i,0} , r_{i,1})$;
	\ENSURE Rectangles $R_1,R_2,R_3$ which are added to the complex such that
	  $|R_1| \cup |R_2| = |R|$ and $R_3$ is the intersection of~$R_1$ and~$R_2$;
	\STATE List of rectangles $L$;
	\FOR { every~$S$ in~$\bdr{R}$ such that~$x$ is in the interior of
	  the projection of~$S$ onto the $i$-th coordinate direction. }
		\STATE Put $S$ into $L$;
	\ENDFOR
	
	\STATE Create a rectangle $R_1 = [r_{1,0},r_{1,1}] \times \ldots \times
	  [r_{i,0},x] \times \ldots \times [r_{n,0} , r_{n,1}]$;
	\STATE Create a rectangle $R_2 = [r_{1,0},r_{1,1}] \times \ldots \times
	  [x,r_{i,1}] \times \ldots \times [r_{n,0} , r_{n,1}]$;
	\STATE Create a rectangle $R_3 = [r_{1,0},r_{1,1}] \times \ldots \times
	  [x,x] \times \ldots \times [r_{n,0} , r_{n,1}]$;
	\STATE Put~$R_3$ into the boundary of~$R_1$ and~$R_2$;
	\FOR { every~$A$ in the boundary of $R$ such that $A \not \in L$ }
		\STATE Remove~$R$ from the coboundary of $A$ and replace it by either~$R_1$ or~$R_2$;
		\STATE Put~$A$ into the boundary of~$R_1$ or~$R_2$ wherever it belongs;
	\ENDFOR
	\FOR { every~$B$ in the coboundary of~$R$}
		\STATE Remove~$R$ from boundary of~$B$ and replace it by both~$R_1$ and~$R_2$;
		\STATE Put~$B$ into the coboundary of both~$R_1$ and~$R_2$;
	\ENDFOR
	\STATE Remove~$R$ from the complex;
	\FOR { every $S \in L$ }
		\STATE $(S_1,S_2,S_3) = SubdivideRectangle( S,i,x );$
		\STATE Put~$S_\ell$ into the boundary of~$R_\ell$, and~$R_\ell$ into
		the coboundary of~$S_\ell$ for $\ell \in \{1,2,3\}$;
	\ENDFOR
  \end{algorithmic}
\end{algorithm}
\begin{algorithm}[tb]
  \small
  \caption{Rectangle subdivision in all coordinate directions.}
  \label{alg:subdividingRectangleAllDirectionsAlgorithm}
  \begin{algorithmic}
	\REQUIRE $R = [r_{1,0},r_{1,1}] \times \ldots \times [r_{n,0} , r_{n,1}]$, \;
	  $x_i \in (r_{i,0} , r_{i,1})$ for $i \in \{1,\ldots,n\}$;
	\ENSURE Collection of rectangles obtained after subdividing~$R$ in
	  all coordinate directions;
	\STATE List of rectangles $L$;
	\STATE $L \leftarrow R$;
	\FOR { $dim = 1$ to $n$ }
			\STATE List of rectangles $L'$;
			\FOR { Every rectangle $R$ in the list $L$ }
				\STATE Call Algorithm~\ref{alg:subdividingRectangleAlgorithm}
				  for $R$, $dim$ and $x_{dim}$, which returns the top-dimensional
				  rectangles~$R_1$ and~$R_2$;
				\STATE $L' \leftarrow R_1, R_2$;
			\ENDFOR
			\STATE $L = L'$;
	\ENDFOR
	\STATE Return all the rectangles obtained in the last iteration of the loop;
  \end{algorithmic}
\end{algorithm}
\begin{figure}[tb]
  \centering
  \includegraphics[width=14cm]{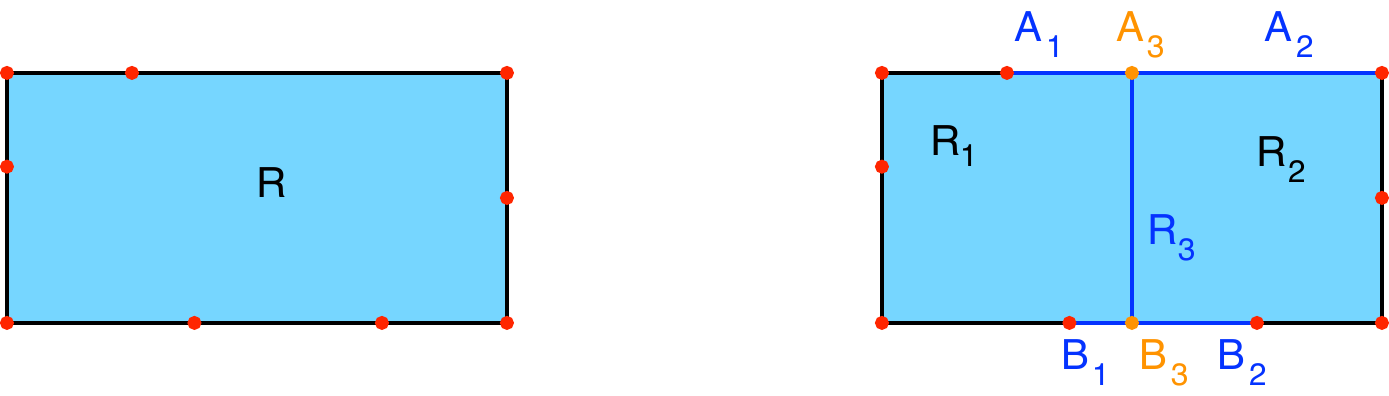}
  \caption{Illustration of Algorithm~\ref{alg:subdividingRectangleAlgorithm}.
           Given a two-dimensional rectangle~$R$ as shown on the left, as
           well as a coordinate direction~$i$, the rectangle is first divided
           into two two-dimensional rectangles~$R_1$ and~$R_2$ by subdividing
           the edges of~$R$ parallel to the $i$-th coordinate direction.
           This also creates their common face~$R_3$. Subsequently, the
           boundary elements of~$R$ which are parallel to the $i$-th
           coordinate direction and which contain the subdivision coordinate 
           in their interior are processed by recursive calls of the
           algorithm. This creates the new one-dimensional rectangles~$A_1$,
           $A_2$, $B_1$, and~$B_2$, and the new zero-dimensional rectangles~$A_3$
           and~$B_3$. Once all of these have been created, the boundaries and
           coboundaries are updated accordingly.}
  \label{fig:rectangleDivisionExample}
\end{figure}

To close this brief section, we would like to point out that the use
of rectangular CW-complexes as described above for the computation of
persistent homology is only one possible approach. Alternatively, one
could construct a simplicial complex directly from the collection of
top-dimensional cells obtained by the approximation algorithm from the
last section. Since all rectangles considered in our approach are convex
sets, one can construct a \emph{weighted nerve complex} which inherits
the (multi-dimensional) persistent homology from the original rectangular
CW-complex associated with the lower semi-continuous approximation of the
underlying continuous function. This is a consequence of the celebrated
nerve theorem, see~\cite{herbert} for more details. The nerve complex can
be constructed as follows. Its vertices are given by the top-dimensional
rectangles in the $\epsilon$-approximation. Moreover, a set of~$k+1$ of these
vertices forms a $k$-dimensional simplex, if the associated top-dimensional
cells have nonempty intersection. Clearly, in the situation of this paper
which is based on the subdivision of rectangles, for a planar set
only~$2^2 = 4$ rectangles can have nonempty intersections. For rectangular
sets embedded in~$\mathbb{R}^3$ one can have at most~$2^3 = 8$ intersections,
and in the general case of rectangles in~$\mathbb{R}^d$ the number of
intersections is bounded by~$2^d$, i.e., it is of exponential complexity
with respect to the embedding dimension~$d$. We would like to point out,
however, that it suffices to include only simplices up to the embedding
dimension~$d$, which leads to a significant reduction in the size of the
nerve complex without changing its associated persistent homology. From
a computational perspective, it is not immediately clear whether the
nerve complex leads to a more memory efficient representation than the
rectangular CW-complex constructed via the subdivision algorithm. Nevertheless,
we have decided to use the rectangular CW-complex approach, as it reflects
the function approximation. In a sense, the lower-dimensional rectangles
can be thought of as transition elements between one function value and
another. In addition, our approach makes it easier to determine persistent
homology in the case of periodic boundary conditions. For this, one only
has to impose the periodic boundary conditions on the initial CW-complex
representation of the domain --- and the presented subdivision algorithms
automatically lead to the correct persistent homology, as long as this
initial CW-complex configuration contains at least two top-dimensional
rectangles in each coordinate direction.
\subsection{Rigorous Persistence Approximation}
\label{sec:cubicalApproximationPersistence}
In this section, we combine the $\epsilon$-approximation algorithm
presented in Section~\ref{sec:approximationContFunctions} with the
rectangle subdivision scheme presented in Section~\ref{sec:SubdivAlgorithm}.
This allows us to construct a filtered regular CW-complex which
rigorously approximates a given continuous function~$f$ on a rectangular
domain~$R$, and at the same time provides a rectangular CW-complex
representation of~$R$ compatible with the filtering function. The
construction will be performed in such a way that the (multi-dimensional)
persistence of the final rectangular CW-complex can easily be computed,
and that the corresponding persistence diagram deviates from the one
for the underlying function~$f$ by less than~$\epsilon$ with respect
to the bottleneck distance. The detailed procedure is listed in
Algorithm~\ref{alg:computePersistenceOfAFunction}, where the initial
rectangle~$R$ is subdivided according to either of the algorithms
contained in Section~\ref{sec:SubdivAlgorithm} until the rigorous
enclosure~$\hat{f}(Q)$ of the range of~$f$ over a rectangle~$Q$ deviates
from the midpoint value by at most~$\epsilon$. When this 
criterion is met for all rectangles~$Q$ in the subdivision of~$R$,
the final CW-complex decomposition is returned together with the
cell values set as the function values of~$f$ at the cell midpoints.
See Algorithm~\ref{alg:computePersistenceOfAFunction} for more details.
\begin{algorithm}[tb]
  \small
  \caption{Rigorous approximation of the persistence of a continuous function.}
  \label{alg:computePersistenceOfAFunction}
  \begin{algorithmic}
    \REQUIRE Compact rectangle $R \subset \mathbb{R}^n$,
	         continuous function $f : R \rightarrow \mathbb{R}$,
	         error bound $\epsilon > 0$.
    \ENSURE Persistence diagram whose distance to the diagram of~$f$ is at most~$\epsilon$.
    \STATE Queue $\mathcal{Q}$;
    \STATE $\mathcal{Q} \leftarrow R$;
    \WHILE {$\mathcal{Q}$ is not empty}
    	\STATE $Q \leftarrow \mathcal{Q}$;
		\STATE $v = midpt(\hat{f}(Q))$;
        \IF { $(max(\hat{f}(Q)) - v \le \epsilon) \;\; {\rm and} \;\; (v - min(\hat{f}(Q)) \le \epsilon$) }
			\STATE $val(Q) = v$;
		\ELSE
			\STATE Subdivide $Q$ in all directions.
			       Put all resulting rectangles in~$\mathcal{Q}$;
		\ENDIF
    \ENDWHILE
    \STATE Determine the lower star filtration of the final complex and compute its
           persistent homology;
    \RETURN Persistence of the resulting rectangular CW-complex;
  \end{algorithmic}
\end{algorithm}

It is clear from our construction, that at the end of the
\emph{while}-loop in Algorithm~\ref{alg:computePersistenceOfAFunction}
one has constructed a lower semi-continuous piecewise constant
function~$\Box f$ of the given continuous function~$f$ as described
earlier in Section~\ref{sec:approximationContFunctions}. We would
like to recall that this approximation satisfies the estimate
\begin{displaymath}
  \left| f(x) - \Box f(x) \right| \le \epsilon
  \quad\mbox{ for all }\quad
  x \in R \; ,
\end{displaymath}
which is guaranteed through the use of rigorous interval arithmetic.
From a computational perspective, the algorithm presented above
returns a rectangular CW-complex~$\mathcal{X}$ in which every
top-dimensional rectangle~$Q$ has been assigned a value~$val(Q)$.
For the computation of persistence, however, we need a fast way to
determine the filtration on~$\mathcal{X}$ which is induced by the
approximating function~$\Box f$. This can be accomplished as follows.
According to our construction, the value of a top-dimensional
rectangle~$Q$ is given by~$val(Q)$. If, on the other hand, we consider
a lower-dimensional rectangle~$S$ in~$\mathcal{X}$, and if $S_1,
\ldots, S_k$ denote all top-dimensional rectangles in~$\mathcal{X}$
which contain~$S$ in their boundary, then we define the value of~$S$
via $val(S) = min\{ val(S_1), \ldots, val(S_k) \}$. These definitions
give rise to a filtration of all rectangles in~$\mathcal{X}$, and it
is not difficult to verify that for every threshold $\alpha \in \mathbb{R}$
the collection of all rectangles~$Q$ with $val(Q) \le \alpha$ is a
CW-subcomplex of~$\mathcal{X}$, whose geometric representation consists
of all $x \in R$ for which $\Box f(x) \le \alpha$. In other words, this
filtration of rectangles provides a fast way to determine the filtration
on~$\mathcal{X}$ induced by~$\Box f$. The above construction is referred
to as the \emph{lower star filtration} on~$\mathcal{X}$ induced by the
values of the top-dimensional rectangles.

Using the lower star filtration on all rectangles one can
now use standard persistence software to determine the associated persistence
diagram. In the case of one-dimensional persistence, it follows from
Theorem~\ref{th:stabilityTheorem} that this persistence diagram has bottleneck
distance at most~$\epsilon$ from the persistence diagram of~$f$ over the
rectangle~$R$. In addition, we would like to point out that
Algorithm~\ref{alg:computePersistenceOfAFunction} also works --- without
significant changes --- for a vector-valued function $f : R \rightarrow
\mathbb{R}^m$. In this case one obtains an approximation~$\Box f$ which
can be used to approximate multi-dimensional persistence. In view of
Theorem~\ref{th:foliationPaper} and similar stability theorems, one
retains the same error bound~$\epsilon$ for the approximation results.
\begin{remark} \label{rem:pers1}
For the above presentation, we have always assumed that the underlying
function~$f$ is at least continuous. The main reason for this is our
requirement that~$f$ has to be amenable to rigorous arithmetic evaluations.
Most available software packages for rigorous arithmetic, such as the
interval arithmetic implementations used throughout this paper, consider
only standard arithmetic operations and standard functions, all
of which are continuous on their domains of definition. Nevertheless,
in principle the above approach can be extended to general \emph{tame
functions}, which ensure that the associated persistence diagrams are
finite. See~\cite{chazal:etal:09a, chazal:etal:12a} for more details.
Note, however, that in general it is not easily possible to verify
whether a given function is tame or not. Moreover, it is possible that
Algorithm~\ref{alg:computePersistenceOfAFunction} terminates for a
non-tame function, provided there are only finitely many persistence
intervals which are longer than~$\epsilon$. Needless to say, this latter
fact is a necessary condition for the termination of
Algorithm~\ref{alg:computePersistenceOfAFunction}.
\end{remark}
\begin{remark} \label{rem:pers2}
In the collection of persistence intervals produced by
Algorithm~\ref{alg:computePersistenceOfAFunction}, there generally
are intervals of length at most~$\epsilon$. Intervals of this type
do not have to correspond to nontrivial intervals in the persistence
diagram of the given function~$f$. For this reason, such intervals will
typically be ignored.
\end{remark}
\subsection{Greedy Approximation}
\label{sec:greedyApproximation}
It was already mentioned earlier that the algorithm presented in
Section~\ref{sec:cubicalApproximationPersistence} does not necessarily
terminate, even if the underlying function~$f$ is Lipschitz continuous.
On the one hand, this can be a consequence of the finite precision
inherent in rigorous computer arithmetic. When a certain subdivision
depth is reached, the range enclosures of~$f$ over a rectangle may not
become smaller even if further subdivisions are performed. On the 
other hand, a given continuous function~$f$ does not necessarily have 
to be tame, i.e., it will not give rise to a finite persistence
diagram.
\begin{algorithm}[tb]
  \small
  \caption{Greedy approximation of the persistence of a function.}
  \label{alg:greedyApproximationOfPersistence}
  \begin{algorithmic}
	\REQUIRE Compact rectangle $R \subset \mathbb{R}^n$,
	         continuous function $f : R \rightarrow \mathbb{R}$.
    \ENSURE Approximate persistence diagram of~$f$, together with an
            upper bound for the bottleneck distance between the approximation
            and the correct diagram.
    \STATE Priority queue $\mathcal{Q}$. The priority of a rectangle~$Q$ is
           defined as the radius $rad(\hat{f}(Q))$;
    \STATE $\mathcal{Q} \leftarrow R$;
    \STATE $Error\_of\_Approximation = rad(\hat{f}(R))$;
    \REPEAT 
		\STATE $Q \leftarrow$ element with highest priority in~$\mathcal{Q}$;
		\STATE { $Error\_of\_Approximation = rad(\hat{f}(Q))$};
		\STATE Subdivide $Q$ in all directions. Put all the resulting rectangles
		    into~$\mathcal{Q}$, ordered by priority;
    \UNTIL{The user interrupts the program or a maximal number of subdivisions is reached};
    \FOR { Every rectangle~$Q$ in the final rectangular CW-complex }
		\STATE $val(Q) = midpt(\hat{f}(Q))$;
    \ENDFOR
    \STATE Determine the lower star filtration and compute persistent homology
           of the resulting complex;
    \RETURN Persistence of the final rectangular CW-complex and the
           $Error\_of\_Approximation$;
  \end{algorithmic}
\end{algorithm}

To accommodate such situations, we close this section with an algorithm
which removes the input parameter~$\epsilon$, inspired by the method
in~\cite{Edels2}. This leads to a procedure which iteratively subdivides
the rectangles in the subdivision of~$R$ in a greedy way until the process
is stopped by the user --- which could be accomplished either by an
interactive interrupt, or by specifying a fixed number of subdivisions
ahead of time. The overall goal of the algorithm is to increase the
precision of the piecewise constant function approximation~$\Box f$
at every subdivision. For this, the algorithm always subdivides the
subrectangle~$Q$ of~$R$ for which the \emph{radius} of the range
enclosure~$\hat{f}(Q)$ is maximal, i.e., the regions over which
the range of~$f$ has the largest variation are subdivided first.
In this context, the radius of an interval~$I \subset \mathbb{R}$
is defined by
\begin{equation} \label{def:radI}
  rad(I) =
  max\left\{ max(I) - midpt(I) , \; midpt(I) - min(I) \right\} \; ,
\end{equation}
and this definition does provide an upper bound on the actual
radius in computer arithmetic implementations.\footnote{Notice that
in infinite precision, the two members of the set on the right-hand
side of~(\ref{def:radI}) have to be the same, i.e., the radius of the
interval is exactly half its width. However, when using finite arithmetic
on a computer, either or both of these numbers might not be a representable
number, which leads to interval answers. In this case, the definition
in~(\ref{def:radI}) ensures that~$rad(I)$ contains a rigorous and tight
upper bound on the radius of the interval.} The
whole procedure can easily be implemented using a priority queue, and
upon termination, the algorithm returns a rigorous upper bound on
the approximation error which is exactly the largest of the range
enclosure radii. The precise implementation of the algorithm can
be found in Algorithm~\ref{alg:greedyApproximationOfPersistence},
and it furnishes a piecewise constant approximation~$\Box f$ for the
given continuous function~$f$ defined on the rectangle~$R$, together
with the rigorous error bound $Error\_of\_Approximation$. This
algorithm can trivially be extended to the case of vector-valued
functions~$f: R \rightarrow \mathbb{R}^m$.
\section{Numerical Case Studies}
\label{sec:experiments}
In this final section of the paper we present a few case studies
which demonstrate how the proposed algorithms perform in practice.
The numerical experiments include a one-dimensional continuous,
but not everywhere differentiable function, a two-dimensional
sample function which serves as a standard benchmark in nonlinear
optimization, as well as an example from mathematical materials
sciences. An implementation of the algorithms used to perform
these computations will be made available upon acceptance of
this paper. It can be found at \url{https://github.com/pdlotko/FunTop}, and is available under the
GPLv3 license.
\subsection{Two Benchmarks from Optimization}
As a first example of the algorithms in this paper we consider
the function
\begin{displaymath}
  f(x) = \frac{\left| \sin(6 \pi x) \right|}{1 + x^2} +
         \frac{3 \cos(2 \pi x)}{10}
  \quad\mbox{ for }\quad
  x \in R = [0,1] \; .
\end{displaymath}
This function is clearly continuous, but not everywhere
differentiable. While it is not possible to use the methods
of~\cite{miro} to approximate the persistence of~$f$, our 
Algorithm~\ref{alg:computePersistenceOfAFunction} can be
applied. The results of this method for $\epsilon = 0.18$,
$\epsilon = 0.06$, and $\epsilon = 0.02$ were already presented
in Figure~\ref{fig:cornerex}. In these persistence diagrams, the 
dashed red line indicates all points above the diagonal whose
maximum norm distance to the diagonal is exactly equal to~$\epsilon$.
In view of Remark~\ref{rem:pers2}, this means that only in the upper
right persistence diagram, which corresponds to $\epsilon = 0.18$,
one of the generators cannot be established rigorously. However,
for the two smaller $\epsilon$-values all generators have been 
proven to exist. They indicate that the function has seven local
minima, that these minima start appearing in a stratified way with
respect to increasing function values, first one, and then in three
pairs of two, and that between any two consecutive minima a local
maximum annihilates homology components. In this sense, our algorithms
can be easily used to rigorously solve one-dimensional optimization
problems.

But the methods are by no means restricted to one space dimension.
Consider for example the two-dimensional function
\begin{equation} \label{def:ackley1}
  f(x,y) = g(45 x - 15, 45 y - 15) \; ,
\end{equation}
where the function~$g$ is the standard two-dimensional Ackley
function given by
\begin{equation} \label{def:ackley2}
  g(x,y) = a + e - a e^{-b \sqrt{(x^2 + y^2) / 2}} -
           e^{(\cos(c x) + \cos(c y)) / 2} \; .
\end{equation}
For our simulations, we choose the parameters $a = 20$, $b = 1/5$,
and $c = \pi / 2$, and we consider the function~$f$ on the rectangle
$R = [0,1]^2 \subset \mathbb{R}^2$. A partially constant approximation
of this function has already been shown in Figure~\ref{fig:introex},
and the actual function~$f$ is shown in Figure~\ref{fig:ackley2D}.
\begin{figure}[tb]
  \centering
  \includegraphics[width=14cm]{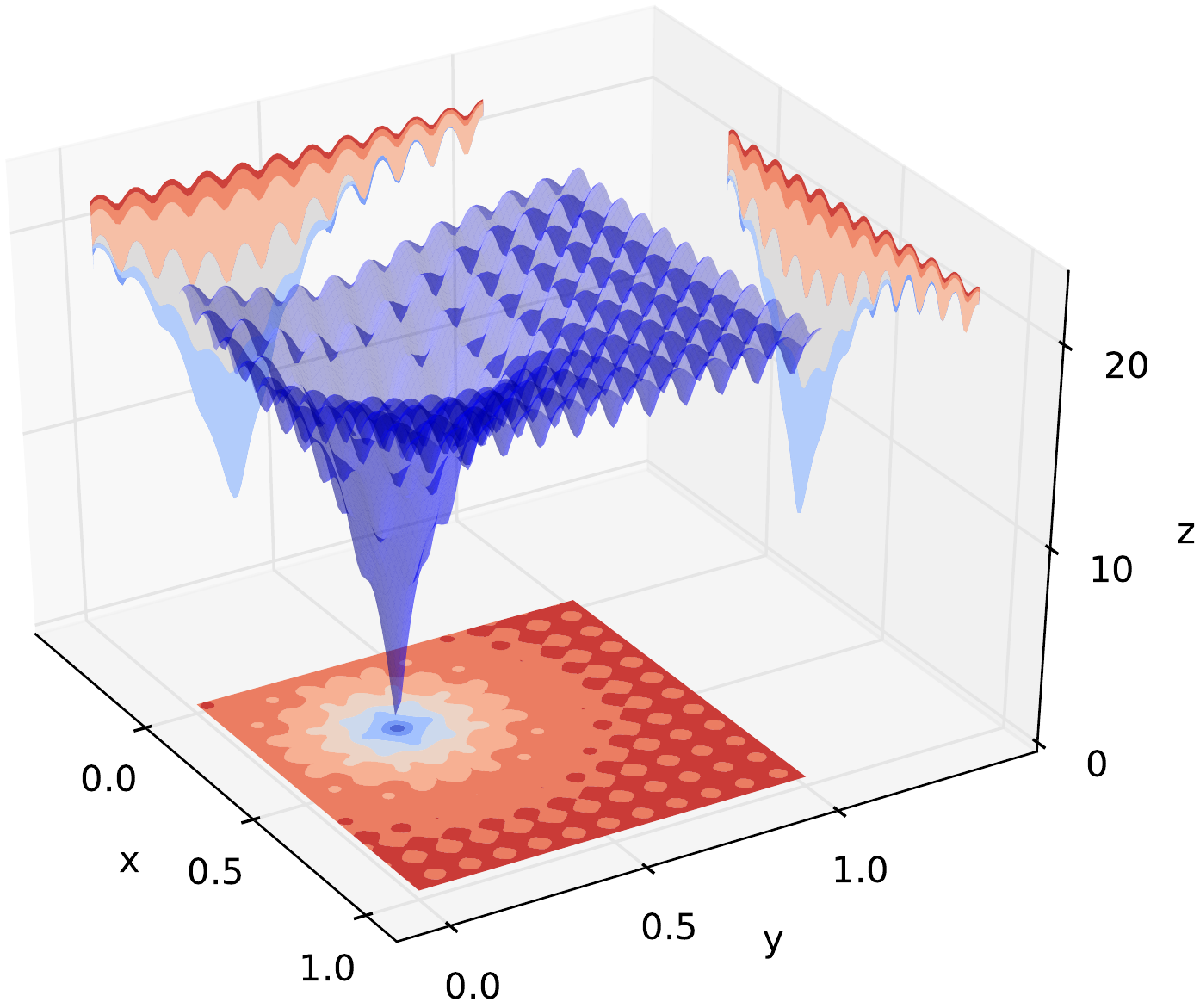}
  \caption{Visualization of the scaled Ackley function in two
           dimensions on the domain $[0,1]\times[0,1]$. The figure
           shows the graph of the function~$f$ defined
           in~(\ref{def:ackley1}), which is based on the standard
           two-dimensional Ackley function~$g$ given in~(\ref{def:ackley2}).
           The function~$g$ is one of the basic benchmark functions in
           nonlinear optimization due to its unique global minimum
           at~$(0,0)$, combined with a large number of local minima.
           Recall that a rigorous piecewise constant approximation~$\Box f$
           for $\epsilon = 2$ is shown in the right panel of
           Figure~\ref{fig:introex}.}
  \label{fig:ackley2D}
\end{figure}
\begin{table}
  \centering
  \begin{tabular}{| c | c | c | c | c | c | c |}
  \hline
  $\epsilon$ & 1      & 0.75    & 0.5     & 0.4     & 0.3     & 0.25  \\
  \hline
  time    &0.003487 & 0.008145 & 0.018178 & 0.027688 & 0.029342 & 0.132159 \\
  \hline 
  $\epsilon$ & 0.2  & 0.1 & 0.075 & 0.05 & 0.025 & 0.02\\
  \hline 
  time & 0.203003 & 2.17143 & 6.62386 & 38.8436  & 632.568 & 1704.78\\
  \hline
  \end{tabular}
  \caption{Computation times for rigorously determining the persistent
           homology of the function~$f$ shown in Figure~\ref{fig:ackley2D},
           for a variety of values of the parameter~$\epsilon$.}
  \label{table:ackley2D}
\end{table}

For the above function~$f$ we determined rigorous
$\epsilon$-approximations of the associated persistence
diagram, for a variety of $\epsilon$-values. The results 
from these computations can be found in Table~\ref{table:ackley2D}.
These computational times seem to exhibit quadratic growth
in~$1 / \epsilon$, as would be expected from our
Theorem~\ref{thm:complexity}. In fact, simple quadratic
regression implies a near perfect fit of the computational
data with the polynomial $T(s) = 71.94 s^2 - 25.36 s +1.11$,
and with correlation coefficient~$0.9869$.
\subsection{The Diblock-Copolymer Model}
Topological methods have seen a wide variety of applications
in the physical sciences, see for example the articles in the
special issue~\cite{day:etal:16a}. We therefore close this 
paper with a related application of our methods to the 
rigorous approximation of persistence in an example from
materials science. We consider the so-called \emph{diblock
copolymer model} which is the fourth-order parabolic 
partial differential equation given by
\begin{eqnarray}
  u_t & = & -\Delta \left( \epsilon^2 \Delta u + F(u) \right)-
    \sigma(u-\mu)
    \quad\mbox{ in }\; \Omega
    \; , \label{appl:dbcp:eqn} \\[1ex]
  & & \mu = \frac{1}{|\Omega|} \int_{\Omega} \! u(x) \, dx \; ,
    \quad\mbox{ and }\quad
    \frac{\partial u}{\partial \nu} =
    \frac{\partial \Delta u}{\partial \nu} = 0
  \quad\mbox{ on }\; \partial \Omega \; ,
\nonumber
\end{eqnarray}
which is an evolution equation for the unknown
function~$u : \R_0^+ \times \Omega \to \R$, where $\Omega
\subset \R^d$ denotes an arbitrary domain with sufficiently
smooth boundary. In this formulation, the function~$F$ is
the negative derivative of a double-well potential, and 
one usually considers $F(u) = u - u^3$. Moreover, the directional
derivative~$\partial u / \partial\nu$ denotes the derivative of~$u$
in the direction of the outward unit normal vector at a point
on the boundary~$\partial\Omega$.
\begin{figure}[tb]
  \centering
  \includegraphics[width=4.5cm]{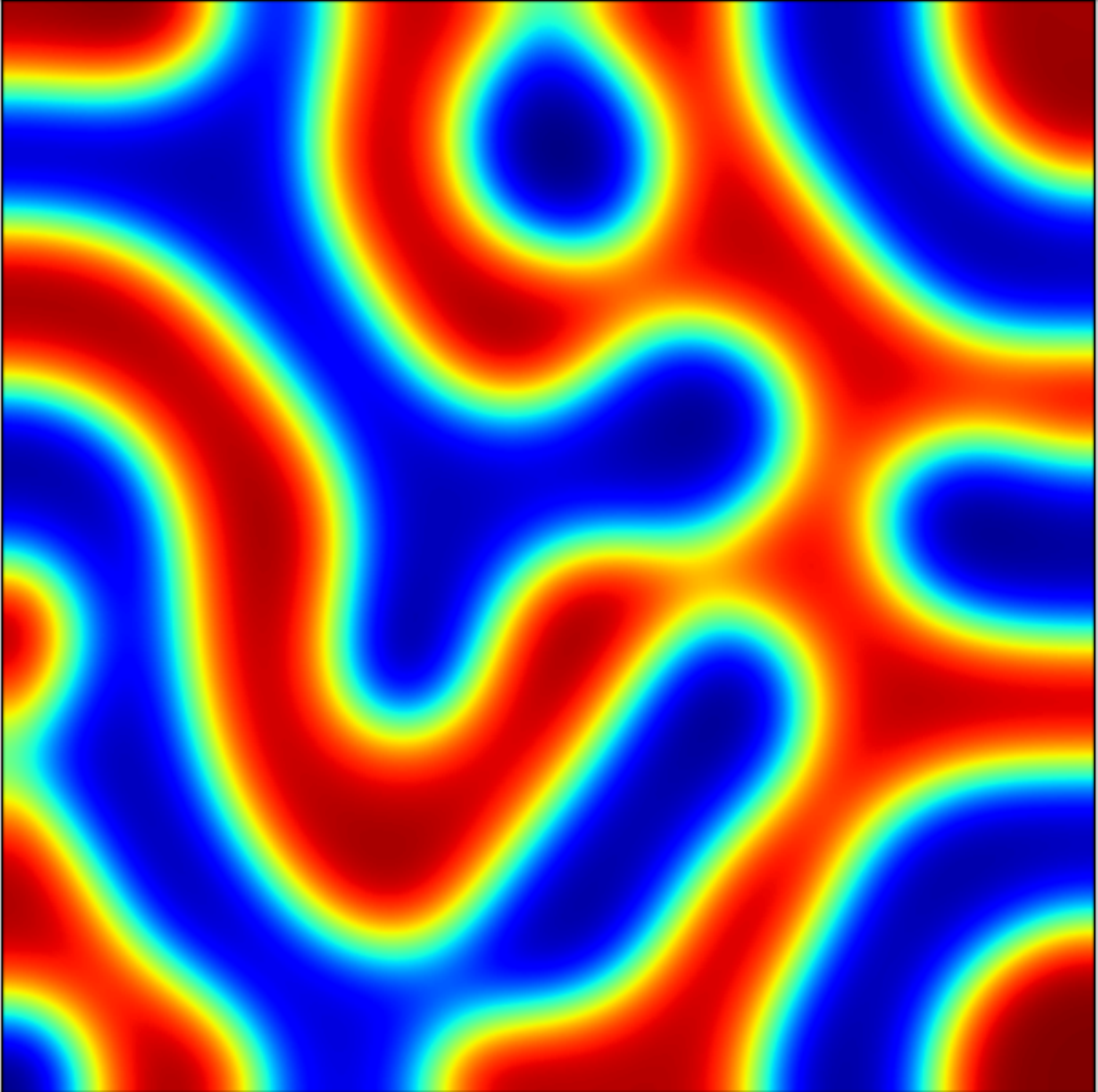}
  \hspace{0.3cm}
  \includegraphics[width=4.5cm]{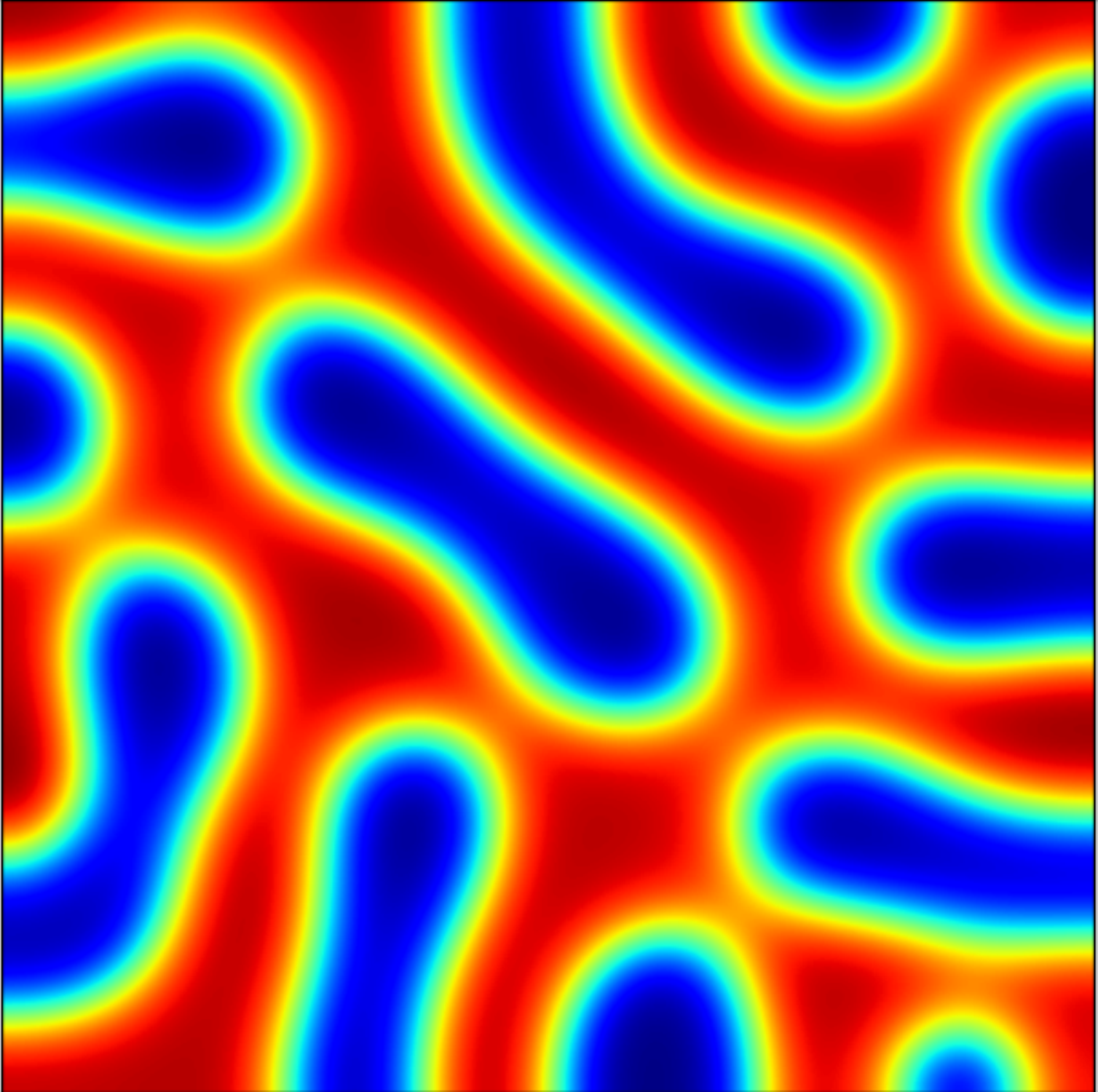}
  \hspace{0.3cm}
  \includegraphics[width=4.5cm]{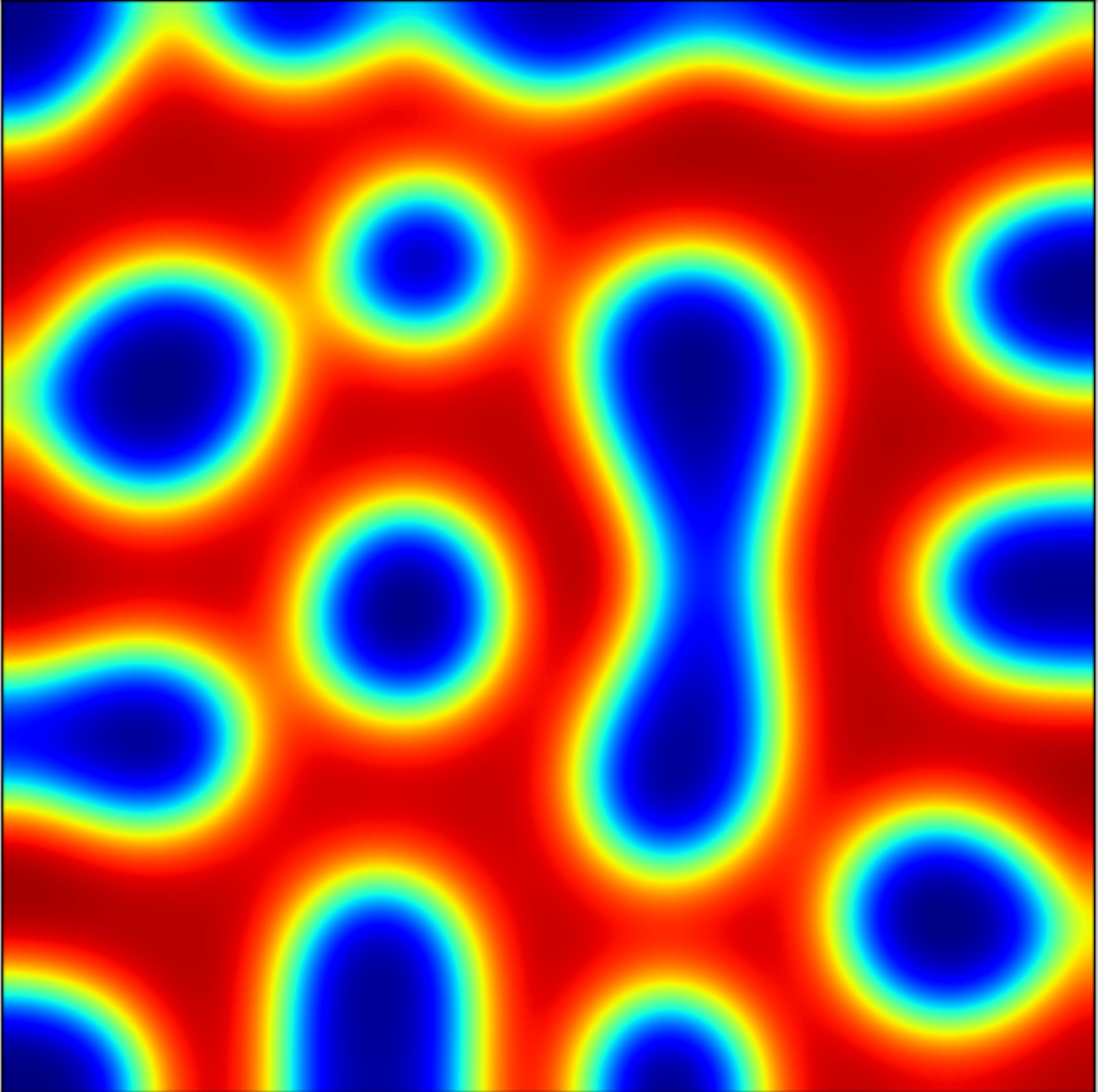} \\[3ex]
  \includegraphics[width=4.5cm]{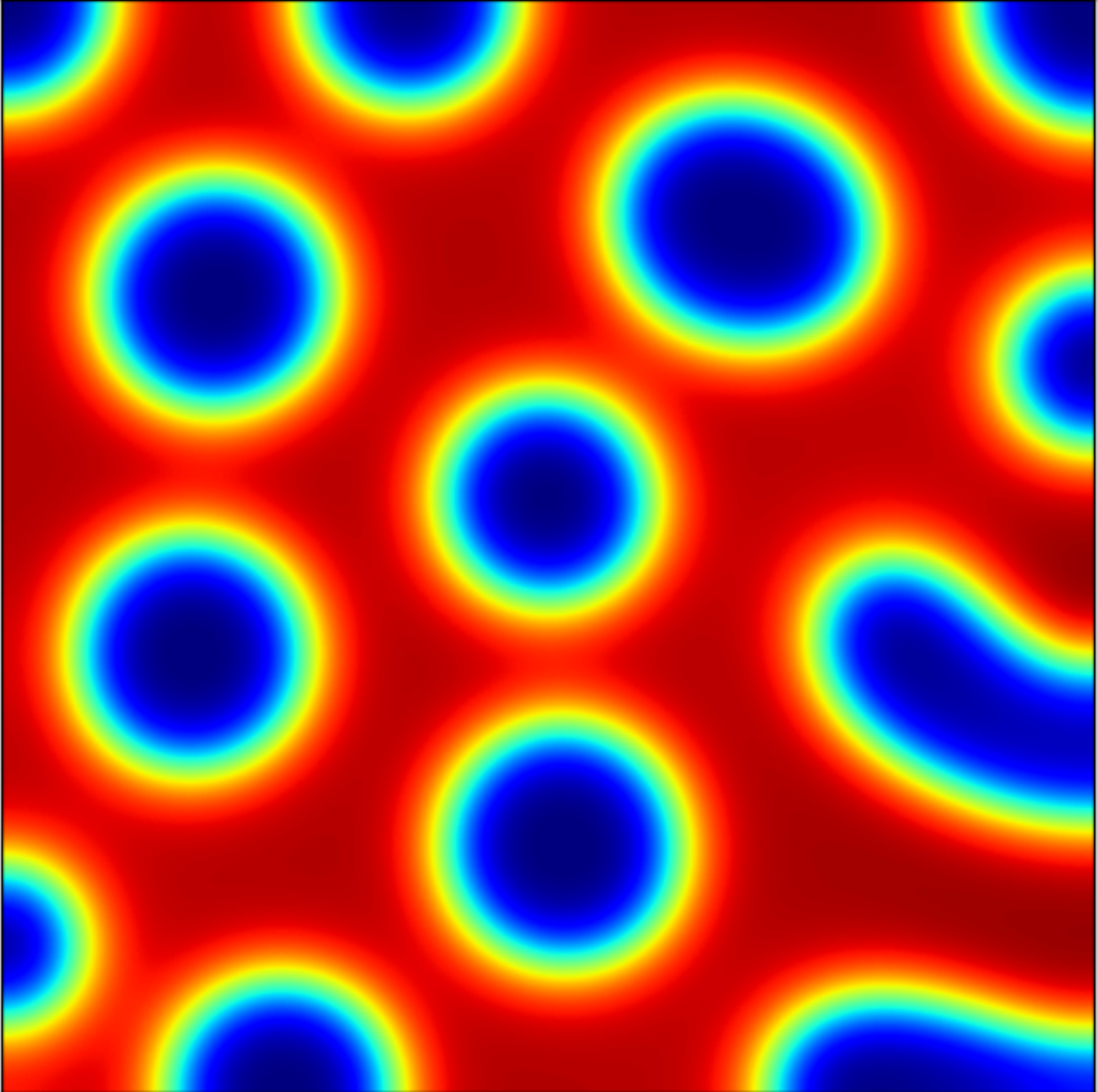}
  \hspace{0.3cm}
  \includegraphics[width=4.5cm]{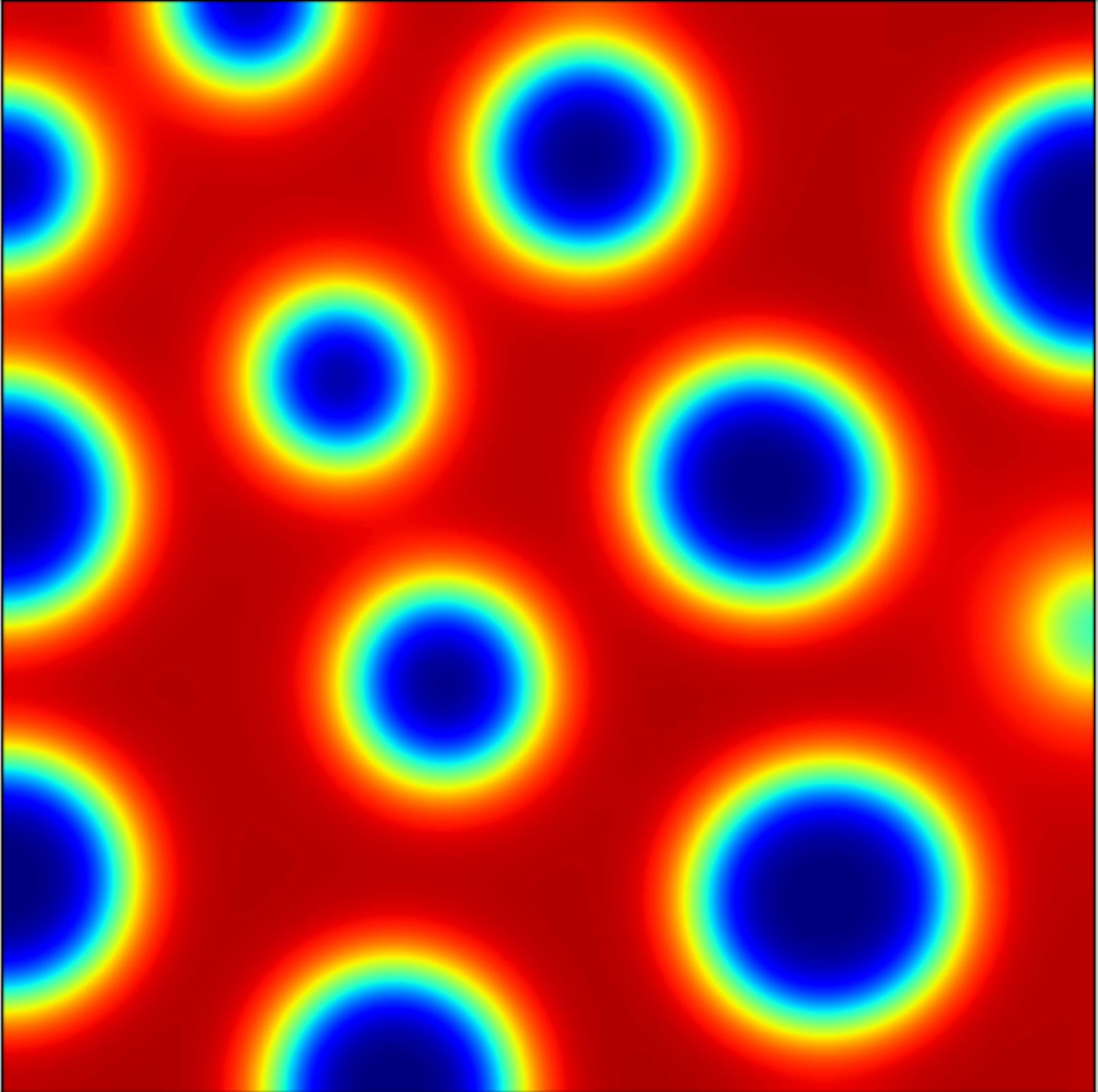}
  \hspace{0.3cm}
  \includegraphics[width=4.5cm]{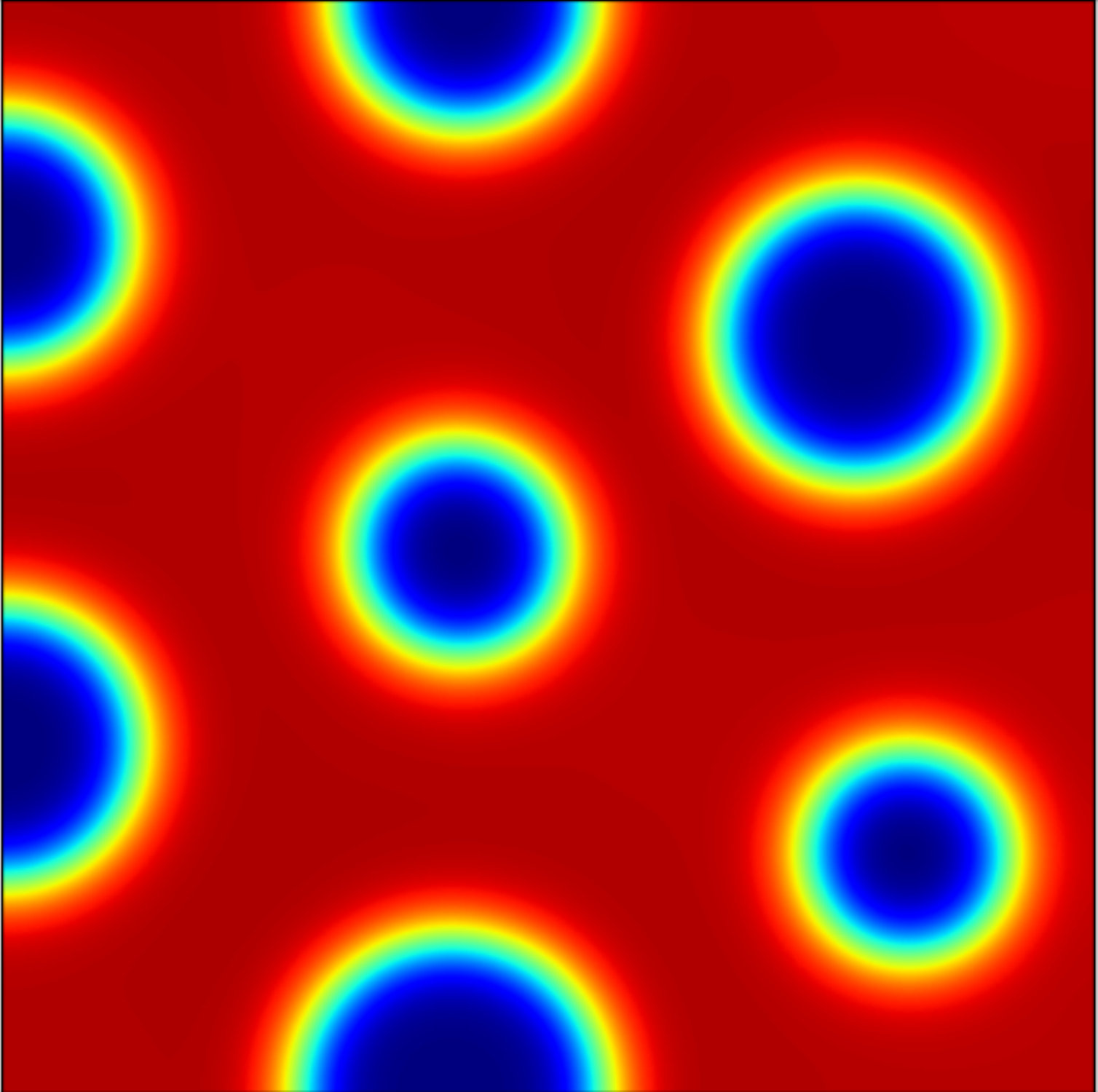}
  \caption{Sample patterns produced by the diblock copolymer
           model~(\ref{appl:dbcp:eqn}) on the two-dimensional square
           domain $\Omega = (0,1)^2$. The images are all for the parameter
           values $\epsilon = 0.025$ and $\sigma = 10$, and for
           end times which correspond to the formation of distinct
           phases. From top left to bottom right the images are
           for mass values $\mu = 0, \; 0.1, \; \ldots, \; 0.5$.}
  \label{fig:dbcppattern}
\end{figure}

While a more complete introduction to the model can be found
in~\cite{wanner:16a}, we only focus on a few important aspects.
The partial differential equation~(\ref{appl:dbcp:eqn}) models
phase separation phenomena in polymer materials which are composed
of two chemically incompatible monomers, say of type~A and type~B,
which are arranged in long polymer chains. The value of the phase
variable~$u$ describes the local material composition in the
following way. Function values of the solution~$u(t,x)$ which
are close to~$+1$ are interpreted as only monomer~A being present
near point $x \in \Omega$ and at time~$t \ge 0$, and the value~$-1$
indicates that only monomer~B is present. Finally, values in
between correspond to mixtures of the two components, with zero
representing an equal mixture. The parameter~$\mu$ denotes the
average mass of the mixture using the same convention, and the two
remaining parameters~$\epsilon > 0$ and~$\sigma \ge 0$ are dimensionless
interaction lengths. Informally, $\epsilon > 0$ being small
corresponds to short range repulsions being strong, inducing a strong
compulsion to separate, while~$\sigma$ being large represents strong
long range chain elasticity forces, inducing a strong compulsion
to hold together. Notice that for~$\sigma = 0$ the diblock copolymer
model~(\ref{appl:dbcp:eqn}) reduces to the celebrated
Cahn-Hilliard equation, which serves as a basic model for
the phase separation phenomena.

If one considers solutions~$u$ which start at a small random
perturbation~$u_0$ of the constant initial state~$\mu$, and if one 
assumes that the average value of~$u_0$ equals~$\mu$, then solutions
quickly start to grow in amplitude, until the function values
reach values close to~$\pm 1$. At this point, the phases have 
separated and one can observe a wide variety of patterns, see for
example Figure~\ref{fig:dbcppattern}. One can clearly see from
these images that the topology of the phase separated state depends
on the total mass~$\mu$. In fact, even though different initial
states lead to different microstructures, the observed patterns for
fixed~$\mu$ share the same characteristic features. In other words,
the parameter~$\mu$ clearly affects the observed topology of the
pattern.

In our recent paper~\cite{dwphysicad} we studied a stochastic
version of the Cahn-Hilliard equation to show that in fact the
converse is true as well, i.e., that the topology of the observed 
patterns can be used to determine the underlying total mass~$\mu$
to surprising accuracy. This was accomplished in the following
way. For every mass value, we performed a number of Monte-Carlo
type simulations, to obtain a collection of persistence diagrams
for the resulting microstructures. By equivalently reformulating
the diagrams as persistence landscapes, it is then possible to
determine an averaged persistence landscape, which can be thought
of as the ``typical'' pattern topology for a given $\mu$-value.
If one then considers an arbitrary microstructure, one can determine
the associated mass value by finding the closest averaged landscape.
This classification scheme exhibits surprising accuracy in
determining the parameter~$\mu$.
\begin{figure}[tb]
  \centering
  \includegraphics[width=7.5cm]{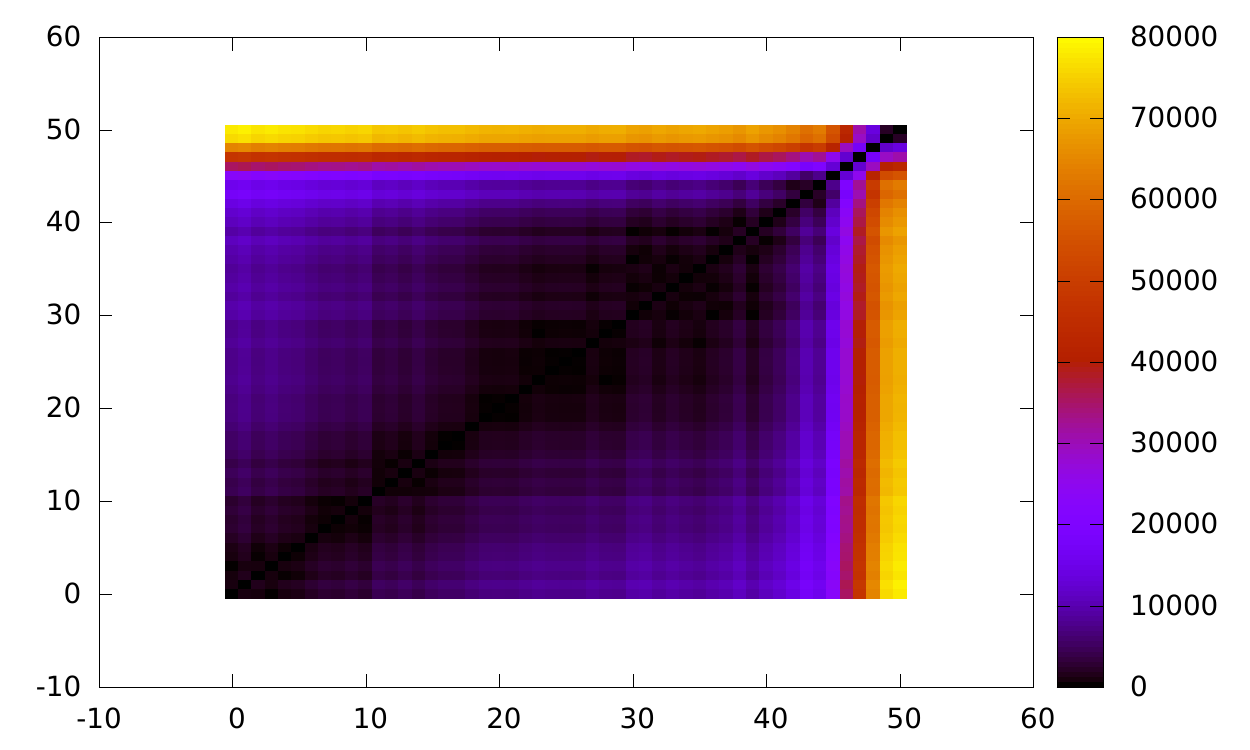}
  \hspace{0.5cm}
  \includegraphics[width=7.5cm]{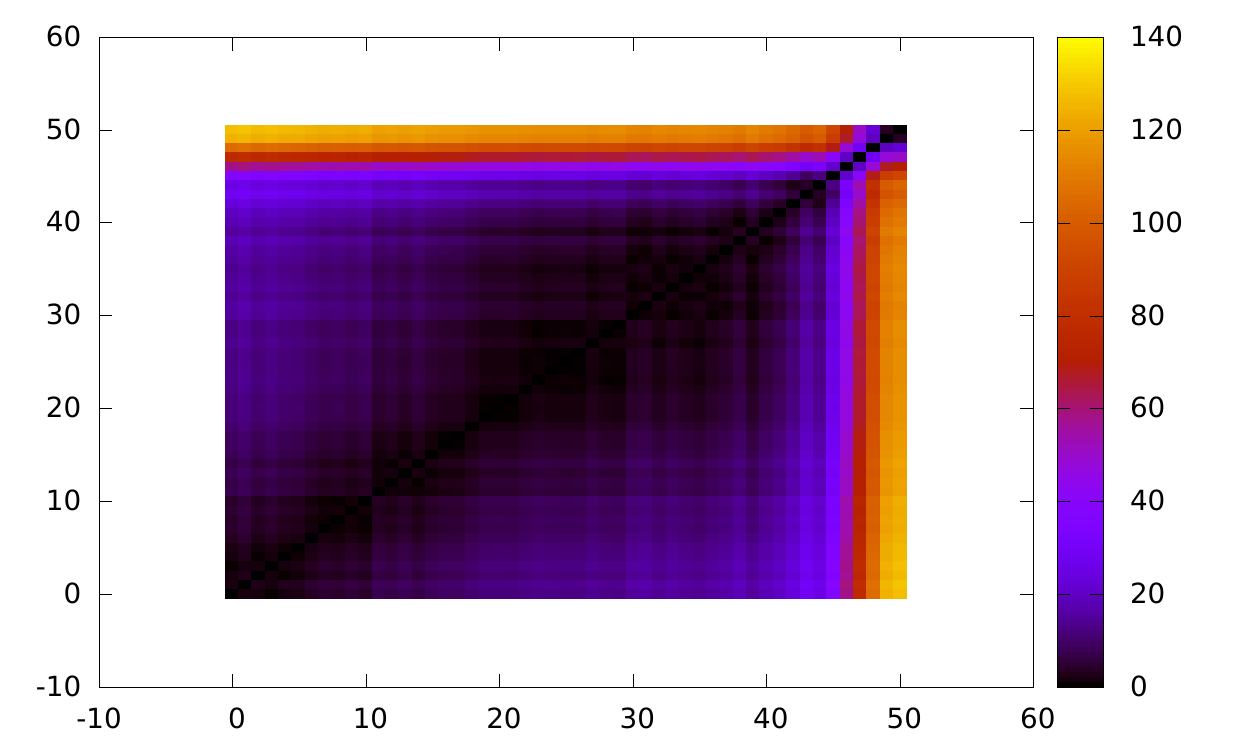} \\[3ex]
  \includegraphics[width=7.5cm]{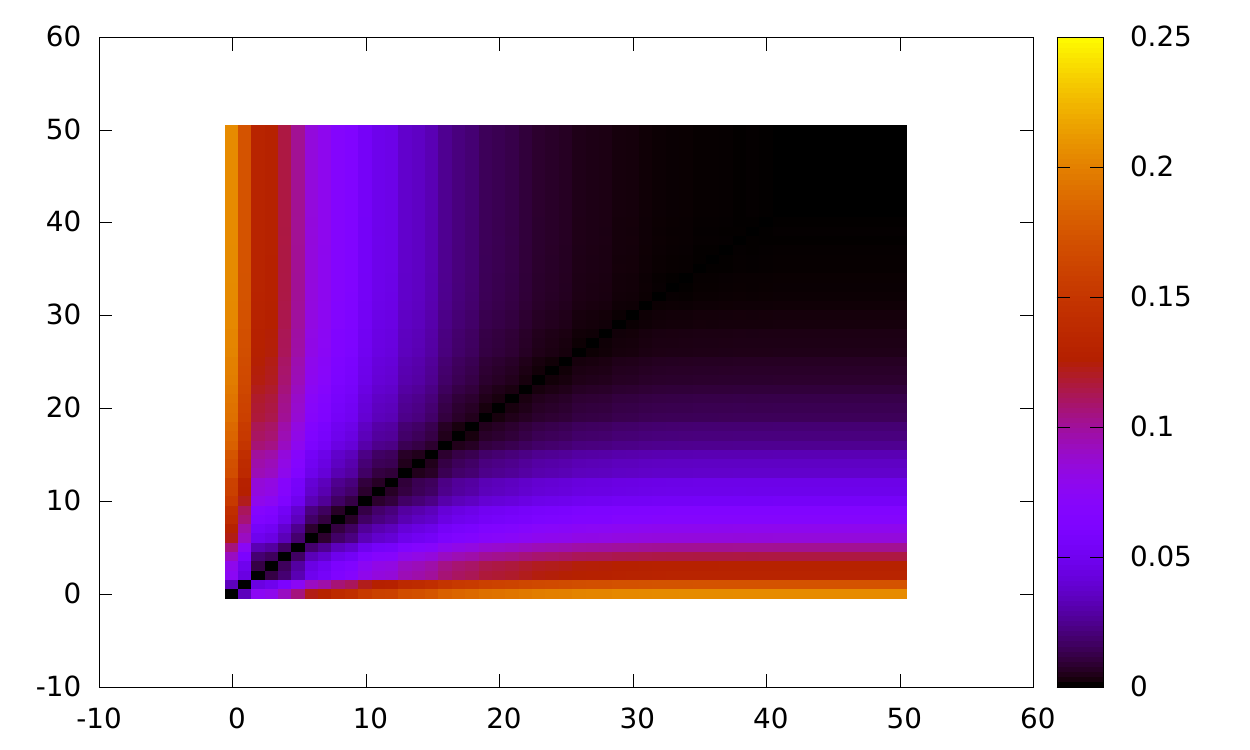}
  \hspace{0.5cm}
  \includegraphics[width=7.5cm]{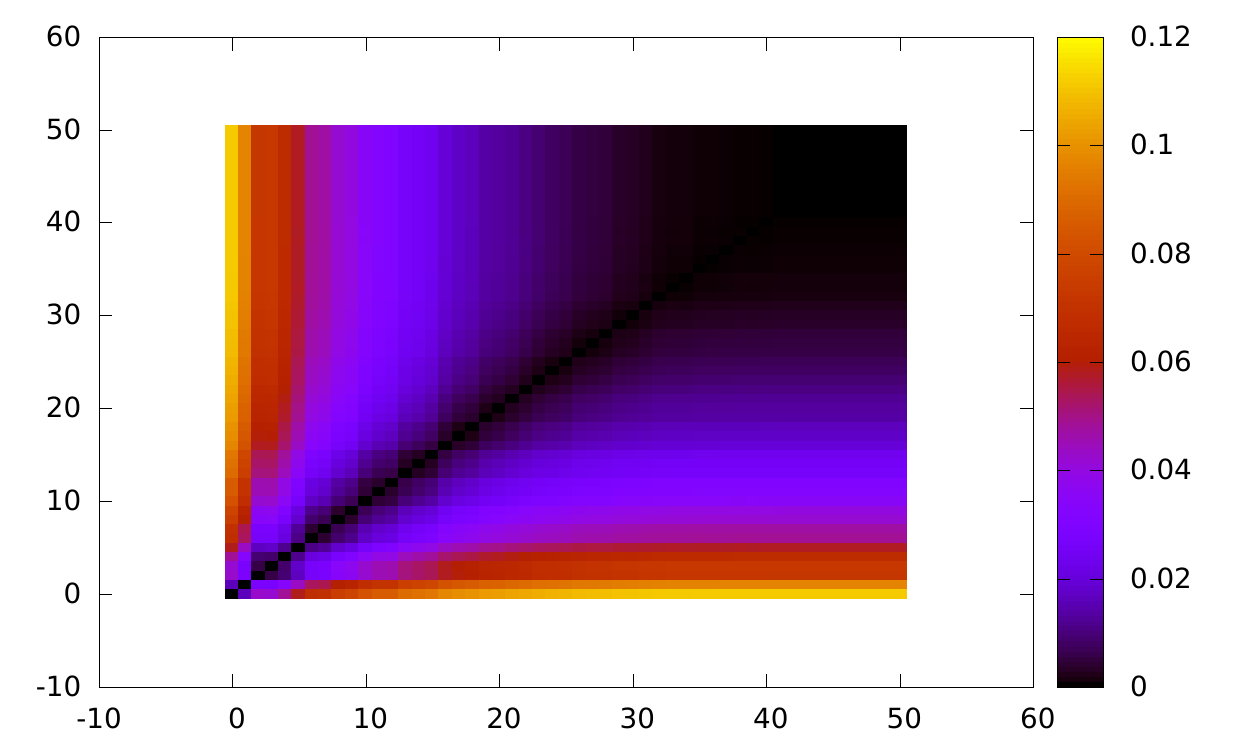}
  \caption{Heat plots of the distance matrices for the averaged
           persistence landscapes of phase separated states in
           the diblock copolymer model~(\ref{appl:dbcp:eqn}).
           The horizontal and vertical axes correspond to~$100\mu$,
           i.e., the images consider mass value~$\mu$ between
           $\mu = 0$ and $\mu = 0.5$ in steps of~$0.01$.
           The top row shows distance matrices for
           the $0$-dimensional persistence landscapes,
           while the bottom row is for dimension~$1$.
           The left column uses the $L^1$-norm to measure
           the distance between the landscapes, the right
           column employs the $L^2$-norm.}
  \label{fig:dbcpDistanceData}
\end{figure}

Rather than performing an analogous study of the diblock copolymer
equation for $\sigma \neq 0$ in full detail, we use the methods
developed in this paper to determine the distance matrices between
the averaged persistence landscapes for each considered mass value.
We consider the same parameter values as in Figure~\ref{fig:dbcppattern},
i.e., we let $\epsilon = 0.025$ and $\sigma = 10$, and consider the
two-dimensional square domain $\Omega = (0,1)^2$. Finally, we simulated
the equation at the~$51$ mass values $\mu = k \cdot 0.01$, with 
multiplier $k = 0,\ldots,50$. In each case, one hundred Monte-Carlo
simulations were performed to determine the averaged persistence
landscape as in~\cite{plt}, from the associated~$100$ persistence
diagrams of phase separated patterns. All of the necessary persistence
diagrams were computed using Algorithm~\ref{alg:computePersistenceOfAFunction} with the error bound set to $0.2$.
The simulations of the partial differential equation employed a
semi-linear spectral method with~$128^2$ eigenmodes for the solution
representation, and~$256^2$ modes for the nonlinearity computations.

The results from these computations can be found in
Figure~\ref{fig:dbcpDistanceData}. This figure contains four
heat plots for the distance matrices of the averaged persistence
landscapes. In each of the plots, the horizontal and vertical axes
correspond to the scaled mass parameter~$100\mu$, i.e., they reflect
total mass values~$\mu$ between $\mu = 0$ and $\mu = 0.5$ in steps
of~$0.01$. The top row contains distance matrices for the
$0$-dimensional persistence landscapes, while the bottom row is
for dimension~$1$. Moreover, the left column uses the $L^1$-norm
to measure the distance between the landscapes, and the right
column employs the $L^2$-norm. These heat maps indicate that
depending on the dimension of the persistence landscape, different
regions of $\mu$-values can be differentiated. Since we are using
nonnegative mass values, for large~$\mu$ one expects a larger
number of persistence intervals, which correspond to the isolated
(negative) droplets in the (positive) background matrix. Thus,
the heat maps in the first row show greater variability for
$\mu \approx 0.5$. In contrast, in the droplet regime, the 
sublevel sets of the function~$u$ are very unlikely to contain
any homology generators in dimension one, and this is reflected
in the large region of distance close to zero in the two bottom
panels of Figure~\ref{fig:dbcpDistanceData}. Nevertheless, for
mass values~$\mu$ close to zero nontrivial homology can be observed
in dimension one, and this leads to the greater variability of the
heat maps in the lower left corner. These observations are very
much in line with the ones in~\cite{dwphysicad}. However, we 
refrain from a more detailed analysis, as it is not the main
topic of the present paper. Nevertheless, these preliminary
results can be considered as a first step towards a quantitative
topological understanding of the topology of microstructures
generated by the diblock copolymer equation.
\section*{Acknowledgments}
The first author of the paper would like to thank Michael Kerber for suggesting
the algorithm presented in the Section~\ref{sec:greedyApproximation}, and Steve
Oudot for all the helpful comments. The first author was supported by the Advanced
Grant of the European Research Council GUDHI (Geometric Understanding in Higher
Dimensions). The second author was partially supported by NSF grants
DMS-1114923 and DMS-1407087.

\end{document}